\documentclass{amsart}
\usepackage{amsmath}
\usepackage{amsthm}
\usepackage{amssymb}
\usepackage{bm}
\usepackage[margin=0.5in]{caption}
\usepackage{dsfont}
\usepackage{enumitem}
\usepackage[style=plain]{floatrow}
\usepackage{graphicx}
\usepackage[hidelinks]{hyperref}
\usepackage[capitalise]{cleveref}
\usepackage{mathrsfs}
\usepackage{mathtools}
\usepackage{nameref}
\usepackage[new]{old-arrows}
\usepackage{stmaryrd}
\usepackage{svg}
\usepackage{thmtools}
\usepackage{xcolor}
\usepackage{xparse}

\usepackage{tikz}
\usetikzlibrary{cd}
\usetikzlibrary{positioning, angles, quotes}
\newcommand{\nospacepunct}[1]{\makebox[0pt][l]{\,#1}} 

\theoremstyle{plain}
\newtheorem{thm}{Theorem}[section]
\newtheorem*{thm*}{Theorem}

\newtheorem{cor}[thm]{Corollary}
\newtheorem*{cor*}{Corollary}

\newtheorem{prop}[thm]{Proposition}
\newtheorem*{prop*}{Proposition}

\newtheorem{lem}[thm]{Lemma}
\newtheorem*{lem*}{Lemma}

\newtheorem{claim}[thm]{Claim}
\newtheorem*{claim*}{Claim}

\newtheorem*{exer*}{Exercise}

\newtheorem*{q*}{Question}

\newtheorem{conj}[thm]{Conjecture}
\newtheorem*{conj*}{Conjecture}

\theoremstyle{definition}
\newtheorem{defn}[thm]{Definition}
\newtheorem*{defn*}{Definition}

\newtheorem{ex}[thm]{Example}
\newtheorem*{ex*}{Example}

\newtheorem{step}{Step}

\theoremstyle{remark}
\newtheorem{rem}[thm]{Remark}
\newtheorem*{rem*}{Remark}

\newtheorem{conv}[thm]{Convention}
\newtheorem*{conv*}{Convention}

\theoremstyle{plain}
\newtheorem{manualtheoreminner}{Theorem}
\newenvironment{manualtheorem}[1]{%
  \renewcommand\themanualtheoreminner{#1}%
  \manualtheoreminner
}{\endmanualtheoreminner}

\newtheorem{manualcorinner}{Corollary}
\newenvironment{manualcor}[1]{%
  \renewcommand\themanualcorinner{#1}%
  \manualcorinner
}{\endmanualcorinner}

\numberwithin{equation}{section}

\Crefname{thm}{Theorem}{Theorems}
\Crefname{defn}{Definition}{Definitions}
\Crefname{claim}{Claim}{Claims}
\Crefname{ex}{Example}{Examples}
\Crefname{prop}{Proposition}{Propositions}
\Crefname{equation}{Equation}{Equations}
\Crefname{figure}{Figure}{Figures}

\newcommand{\BS}{\mathrm{BS}}

\newcommand{\U}{\mathcal U}

\renewcommand{\b}{b^{(2)}}

\newcommand{\Dg}{\mathcal{D}_{K[G]}}
\newcommand{\Dh}{\mathcal{D}_{K[H]}}

\newcommand{\bG}{b^{K[G]}}

\newcommand{\bH}{b^{K[H]}}

\newcommand{\C}{\mathbb{C}}

\newcommand{\Q}{\mathbb{Q}}
\newcommand{\R}{\mathbb{R}}
\newcommand{\Z}{\mathbb{Z}}  
\newcommand{\D}{\mathcal{D}}

\newcommand{\inv}{^{-1}}

\DeclareMathOperator{\cd}{cd}

\DeclareMathOperator{\im}{im}

\DeclareMathOperator{\rk}{rk}

\DeclareMathOperator{\Tor}{Tor}
\newcommand{\ran}{\rangle}
\newcommand{\lan}{\langle}
\renewcommand{\leq}{\leqslant}

\newcommand{\Ga}{\Gamma}
\renewcommand{\subset}{\subseteq}

\newcommand{\lrar}{\longrightarrow}

\newcommand{\po}{\colon}

\newcommand{\vertii}[1]{{\left\vert\kern-0.25ex\left\vert #1 \right\vert\kern-0.25ex\right\vert}}
\newcommand{\vertiii}[1]{{\left\vert\kern-0.25ex\left\vert\kern-0.25ex\left\vert #1 \right\vert\kern-0.25ex\right\vert\kern-0.25ex\right\vert}}

\newcommand{\rchi}{\overline{\chi}}
\newcommand{\E}{\mathsf{Edge}}
\newcommand{\V}{\mathsf{Vert}}
\renewcommand{\o}{\mathsf{o}}
\renewcommand{\t}{\mathsf{t}}

\newcounter{comments}

\title[The Hanna Neumann conjecture for graphs of free groups]{The Hanna Neumann conjecture for graphs of free groups with cyclic edge groups}
\author{Sam P.~Fisher}
\author{Ismael Morales}
\address[S.~P.~Fisher and I.~Morales]{Mathematical Institute, 
Andrew Wiles Building,  
Observatory Quarter, 
University of Oxford, 
Oxford,
OX2 6GG,
United Kingdom}
\email{sam.fisher@maths.ox.ac.uk}
\email{morales@maths.ox.ac.uk}


\begin{document}

\begin{abstract} 
    The Hanna Neumann Conjecture (HNC) for a free group $G$ predicts that $\rchi(U\cap V)\leq \rchi (U)\rchi(V)$  for all finitely generated subgroups $U$ and $V$, where $\rchi(H) = \max\{-\chi(H),0\}$ denotes the \textit{reduced Euler characteristic} of $H$. A strengthened version of the HNC was proved independently by Friedman and Mineyev in 2011. Recently, Antolín and Jaikin-Zapirain introduced the $L^2$-Hall property and showed that if $G$ is a hyperbolic limit group that satisfies this property, then $G$ satisfies the HNC. Antolín and Jaikin-Zapirain established the $L^2$-Hall property for free and surface groups, which Brown and Kharlampovich extended to all limit groups. In this article, we prove the $L^2$-Hall property for graphs of free groups with cyclic edge groups that are hyperbolic relative to virtually abelian subgroups. We also give another proof of the $L^2$-Hall property for limit groups.  As a corollary, we show that all these groups satisfy a strengthened version of the HNC.
\end{abstract}

\maketitle

\section{Introduction}

A group $G$ has the \textit{Howson property} if, for all finitely generated subgroups $U, V\leq G$, the intersection $U \cap V$ is finitely generated. The property is named after Albert G.~Howson, who proved it for free groups in \cite{Howson_fgFree}. Shortly thereafter, Hanna Neumann \cite{HNeumann_fgFree} quantified this property by proving that
\begin{equation*}  
    \rk(U \cap V) - 1 \leqslant 2(\rk(U) - 1)(\rk(V) - 1)
\end{equation*}
whenever $U$ and $V$ are finitely generated subgroups of a common free group, and she conjectured that the factor of $2$ on the right-hand side of the inequality could be dropped. This became known as the {\it Hanna Neumann Conjecture}, and was the beginning of a fruitful line of research concerning these type of inequalities \cite{Dic94, Tar96, Min11}. As we shall see, such inequalities are not limited to free groups. 

Walter Neumann \cite{WNeumann_fgFree} formulated a stronger version of H.~Neumann's conjecture, described in \cref{SHNC}.  Given a group $G$ with a finite classifying space, we denote by $\rchi(G) = \max\{-\chi(G),0\}$ the \textit{reduced Euler characteristic} of $G$. 

\begin{conj}[The Strengthened Hanna Neumann Conjecture (SHNC)] \label{SHNC}
    Let $U$ and $V$ be finitely generated subgroups of a free group $G$. Let $T$ be a complete set of representatives for the double $(U,V)$-cosets in $G$. Then
    \begin{equation} \label{eq:SHNC}
        \sum_{t \in T} \rchi(U \cap V^t) \leqslant \rchi(U) \rchi(V).
    \end{equation}
\end{conj}

The statement of \cref{SHNC} makes sense whenever $G$ is a group such that all of its finitely generated subgroups are of finite type (so that $\overline{\chi}(H)$ is defined for all finitely generated $H \leqslant G$). \Cref{SHNC} was resolved independently by Friedman \cite{Friedman_HN} and Mineyev \cite{Mineyev_HN}. More recently, Jaikin-Zapirain \cite{Jaikin_HNfreeprop} gave an alternative proof which also applies to free pro-$p$ groups $G$. Later on, groups of dimension $2$ were shown to satisfy \cref{SHNC}, such as Demushkin groups by Jaikin-Zapirain--Shusterman \cite{JaikinShusterman_Demushkin} and surface groups by Antolín--Jaikin-Zapirain \cite{Jai22}. An important aspect of the latter article is that the authors introduce the \textit{$L^2$-Hall property} as an intermediate step towards establishing \cref{SHNC} for surface groups. This opened the possibility to showing that the SHNC holds for many more classes of groups.

We briefly recall the $L^2$-Hall property mentioned above before stating our results. Let $\U(G)$ denote the algebra of affiliated operators of a group $G$. Then  $G$ is said to have the $L^2$-Hall property if for all finitely generated subgroups $H\leq G$ there exists a finite-index subgroup $G_1\leq G$ containing $H$ such that the kernel of the corestriction map 
\[
    H_1(H; \mathcal U(G)) \lrar H_1(G_1; \mathcal U(G))
\]
has zero $\U(G)$-dimension (see \cref{def:L2indep,def:L2Hall} for more details). This property is named $L^2$-Hall because of its similarity with the local retractions property which M.~Hall property \cite{Hal49} established for free groups: if $F$ is free and $H$ is a finitely generated subgroup, then there is a finite-index subgroup $G \leqslant F$ containing $H$ and a retraction $G \lrar H$. The local retractions property was extended to surface groups by Scott \cite{Scott_surfLerf} and subsequently to all limit groups by Wilton \cite{WiltonHall}. 

Antolín--Jaikin-Zapirain proved that free and surface groups have the $L^2$-Hall property \cite[Theorem 4.4]{Jai22} and showed that if $G$ is a hyperbolic limit group that has the $L^2$-Hall property, then \cref{SHNC} holds for $G$ \cite[Theorem 1.3]{Jai22}. Recently, Brown and Kharlampovich \cite[Corollary 28]{BrownKar2023quantifying} proved  that the $L^2$-Hall property holds for limit groups  and hence that  \cref{SHNC} holds for  hyperbolic limit groups $G$.

The main result of this article establishes the $L^2$-Hall property for toral relatively hyperbolic graphs of free groups with cyclic edge groups (and hence \cref{SHNC} for these groups (\cref{cor:C})). This is a class of groups that contains not only free groups, surface groups, and some limit groups, but also groups that do not fit into these classes, like the one-relator group with presentation $\lan a, b, c \mid a^2 b^2 c^3 \ran$ (see \cref{introrem}).

\begin{manualtheorem}{A}[\cref{thm:GOFGWCEG}] \label{thm:A}
    Let $G$ be a group splitting as a finite graph of finitely generated free groups with cyclic edge groups. If $G$ is hyperbolic relative to virtually abelian subgroups, then $G$ satisfies the $L^2$-Hall property.
\end{manualtheorem}

We can also prove the $L^2$-Hall property for the class of limit groups.  Perhaps the most famous characterisation of this class is the one confirmed by Sela \cite{Sela06} in his solution of Tarski's problem on classifying finitely generated groups with the same existential theory as a free group. Kharlampovich and Miasnikov also made powerful advances on the structure theory of limit groups, proving that limit groups are exactly the finitely generated subgroups of ICE groups \cite{KharlampovichMyasnikov1998}; this is the smallest class of groups containing all finitely generated free groups that is closed under extending centralisers (\cref{def:ICE}). Wilton \cite{WiltonHall} used this hierarchy in his proof of the local retractions property for limit groups. We build on the methods of Wilton to  establish the $L^2$-Hall property for limit groups in our next result, giving an alternative proof of \cite[Corollary 28]{BrownKar2023quantifying}. The potential interest in revisiting the $L^2$-Hall property for limit groups is to give an inductive argument that could work for more general finite abelian hierarchies (see \cref{conj:specialyHypSHNC} below). 

\begin{manualtheorem}{B}[\cref{thm:L2tame}] \label{thm:B}
   Limit groups satisfy the $L^2$-Hall property.
\end{manualtheorem}

Antolín and Jaikin-Zapirain's proof that the $L^2$-Hall property implies the SHNC for hyperbolic limit groups also applies to toral relatively hyperbolic graphs of free groups with cyclic edge groups and all limit groups. To see this, one needs to incorporate  recent results of Minasyan \cite{Minasyan_WZ} and Minasyan--Mineh \cite{MinasyanMineh_QC} on the Wilson--Zalesskii property and double coset separability, which were not available to Antolín--Jaikin-Zapirain. We review how all these ingredients fit together in \cref{sec:L2SHNC}. Thus, the following is a consequence of Theorems \ref{thm:A} and \ref{thm:B}.

\begin{manualcor}{C}[\cref{cor:HNGFGCEG}] \label{cor:C}
    Suppose that $G$ is either a limit group or that it splits as a finite graph of free groups with cyclic edge groups that is hyperbolic relative to virtually abelian subgroups. Then $G$ satisfies the Strengthened Hanna Neumann conjecture.
\end{manualcor}

The proofs of Theorems \ref{thm:A} and \ref{thm:B} are inspired by Wise's proof of subgroup separability in graphs of free groups with cyclic edge group \cite{Wise_subgroupSep} and Wilton's proof of the local retractions property in limit groups \cite{WiltonHall}, respectively. In the proof of \cref{thm:A}, we make crucial use of the following result at several points.

\begin{manualtheorem}{D}[\cref{overgroups}] \label{thm:D} 
    Let $G$ be a finitely generated locally indicable group with $\cd(G)=2$ and $b_2^{(2)}(G)=0$. Suppose that $G$ has a finite-index subgroup that satisfies the $L^2$-Hall property. Then $G$ satisfies the $L^2$-Hall property.
\end{manualtheorem}

We give one of the reasons why \cref{thm:D} (or, in fact, the stronger version that we prove in \cref{overgroups}) is needed in our proof of \cref{thm:A}. Wise showed in \cite[Theorem 4.18]{Wise_subgroupSep} that subgroup separable (in particular, toral relatively hyperbolic) graphs of free groups with cyclic edge groups have finite-index subgroups that are fundamental groups of \textit{clean} graphs of graphs with $S^1$ edge spaces (here clean means that the edge maps are embeddings). In Wise's argument, it suffices to work with clean graphs of spaces because virtually subgroup separable groups are, again, subgroup separable. However, in general, it is unclear whether the $L^2$-Hall property passes to finite-index overgroups. Thankfully, \cref{thm:D} implies that this is true in our setting. Note that not all subgroup separable graphs of free groups with cyclic edge groups are $L^2$-Hall; for instance, $F_2 \times \Z$ is not $L^2$-Hall.

\begin{rem} \label{introrem} There are conjectures that relate the classes of groups of Theorems \ref{thm:A} and \ref{thm:B}. Wise  asked whether graphs of free groups with cyclic edge groups are virtually limit groups if and only if they do not contain $F_2 \times \Z$ (see \cite[Problem 1.5]{Wis18}). If this was true, then   \cref{thm:D}, together with  the $L^2$-Hall property for limit groups (as proved in \cite{BrownKar2023quantifying} or \cref{thm:B}) would imply \cref{thm:A}. This is the case for the hyperbolic one-relator group $G=\lan a, b, c \mid a^2 b^2 c^3 \ran$, which is a non-limit group that falls under the assumptions of \cref{thm:A}, while it is also virtually limit by \cite{Wis18}.
\end{rem}

It is desirable to have a class of groups satisfying the SHNC containing both the graphs of free groups under consideration and limit groups, as this would provide a unifying framework for our results. We conclude the introduction by proposing such a class. Let $\mathcal C_0$ be the class of groups containing only the trivial group. Inductively, we define $\mathcal C_{n+1}$ to be the class of groups $G$ such that either $G$ is virtually in $\mathcal C_n$ or $G$ has  the form $H *_A$ (resp. $H *_A K$), where $H$ (resp. $H$ and $K$) belong to $\mathcal C_n$ and $A$ is a finitely generated free abelian group. We say that $G$ admits a \textit{finite abelian hierarchy} if it lies in $\mathcal C_n$ for some $n$.

\begin{conj}\label{conj:specialyHypSHNC}
    Let $G$ be a group that admits a finite abelian hierarchy. Suppose that $G$ is torsion-free and hyperbolic relative to  virtually abelian subgroups. Then $G$ is $L^2$-Hall and satisfies the Strengthened Hanna Neumann Conjecture. 
\end{conj}

The inductive proof of \cref{thm:B} already suggests that the argument could carry over into \cref{conj:specialyHypSHNC}. 

\subsection{Organisation of the paper}

In \cref{sec:prelims}, we recall some standard notions that will appear throughout the article, such as graphs of groups (and spaces), group homology, and in particular $L^2$-homology of groups. In \cref{sec:L2hall}, we discuss the $L^2$-Hall property  and discuss both examples and non-examples. The main result of this section is \cref{cor:overgroups}, which gives sufficient conditions to conclude that a group with an $L^2$-Hall subgroup of finite index is itself $L^2$-Hall. This result is crucial in the proof of \cref{thm:A}, which is given in \cref{sec:GofFG}. We prove \cref{thm:B} in \cref{sec:ICE} by adequately modifying Wilton's argument on the local retractions property for these groups. Finally, in \cref{sec:L2SHNC}, we review the  arguments of Antolín--Jaikin-Zapirain to explain how \cref{cor:C} follows from our results combined with recent advances of Minasyan and Mineh on double coset separability.

\subsection{Acknowledgments}

The first author has received funding from the European Research Council (ERC) under the European Union's Horizon 2020 research and innovation programme (Grant agreement No. 850930) and from the National Science and Engineering Research Council (NSERC) [ref.~no.~567804-2022]. The second author is funded by the Oxford--Cocker Graduate Scholarship. 

The authors wish to thank Dawid Kielak and Richard Wade for many helpful suggestions, and Aaron W. Messerla and Henry Wilton for  very useful conversations. The authors would also like to thank  Zachary Munro for generously answering many questions about graphs of free groups with cyclic edge groups.

\section{Preliminaries}\label{sec:prelims}

\subsection{Graphs of groups and spaces}

Graphs of groups were introduced as combinatorial objects in \cite{Serre_arbres}. In \cite{ScottWall_top}, Scott and Wall introduced graphs of spaces in order to study graphs of groups topologically. Since we will use both viewpoints in this article, we take the time to introduce them here.

Throughout this subsection, $\Gamma$ denotes a directed graph, $\V(\Gamma)$ and $\E(\Gamma)$ denote the vertex and edge sets of $\Gamma$, respectively. For any edge $e \in \E(\Gamma)$, let $\mathsf{o}(e) \in \V(\Gamma)$ and $\mathsf{t}(e) \in \V(\Gamma)$ denote the origin and terminus of $e$.

\begin{defn}[Graph of groups] \label{defngraphgroups}
    A \textit{graph of groups} $\mathcal G$ consists of the following data:
    \begin{enumerate}
        \item a connected directed graph $\Gamma$, called the \textit{underlying graph} of $\mathcal G$;
        \item groups $G_v$ and $G_e$ for every vertex $v \in \V(\Gamma)$ and edge $e \in \E(\Gamma)$;
        \item monomorphisms $\varphi_{e,\o} \colon G_e \lrar G_{\o(e)}$ and $\varphi_{e,\t} \colon G_e \lrar G_{\t(e)}$ for every edge $e \in \E(\gamma)$.
    \end{enumerate}
    The groups $G_v$ and $G_e$ are called the \textit{vertex groups} and \textit{edge groups} of $\mathcal G$. The monomorphisms $\varphi_{e,\o}$ and $\varphi_{e,\t}$ are called the \textit{edge maps} of $\mathcal G$. 
\end{defn}

We now review two ways to look at the fundamental group of a graph of groups.  

\begin{defn}[Based fundamental group]
    With the same notation as in  \cref{defngraphgroups}, let $v_0 \in \V(\Gamma)$ be a base vertex and for each $e \in \E(\gamma)$ introduce the formal symbol $t_e$. Let $P(\mathcal G)$ be the group freely generated by the vertex groups $G_v$ and the symbols $t_e$ subject to the relations $t_e \varphi_{e,\t}(g) t_e\inv = \varphi_{e,\o}(g)$ for $e \in \E(\Gamma)$ and $g \in G_e$. The \textit{fundamental group} of the graph of groups based at $v_0$, denoted $\pi_1(\mathcal G, v_0)$, is the subgroup of $P(\mathcal G)$ consisting of the elements that can be represented as words $g_0 t_{e_1}^{\varepsilon_1} g_1 \cdots t_{e_n}^{\varepsilon_n} g_n$, where $g_i \in G_{\t(e_i)}$, where
    \[
        \varepsilon_i = \pm 1, \quad \begin{cases}
            g_i \in F_{\t(e_i)} & \text{if} \ \varepsilon_i = 1 \\
            g_i \in F_{\o(e_i)} & \text{if} \ \varepsilon_i = -1
        \end{cases},
        \quad g_0, g_n \in G_{v_0}
    \]
    and where $(e_1, \dots, e_n)$ forms a (not necessarily directed) loop.
\end{defn}

An equivalent way to look at the fundamental group is explained in \cref{defn:tree_pi1}. This will be used when considering splittings of $\mathcal G$ over simpler subgraphs of groups as in \cref{prop:monster}.

\begin{defn}[Fundamental group relative to a spanning tree]\label{defn:tree_pi1}
    With the notation of \cref{defngraphgroups}, let $T$ be a spanning tree of $\Ga$. The \textit{fundamental group} of $\mathcal G$ relative to $T$, denoted by $\pi_1(\mathcal G, T)$, is the group freely generated by the groups $G_v$ for all $v\in \V(\Ga)$, and the formal symbols $t_e$ for all $e\in \E(\Ga)$, subject to two types of relations: $t_e\phi_{e, \o}(x)t_e^{-1}=\phi_{e, \t}(x)$ for all $e\in \E(\Ga)$ and $x\in G_e$; and $t_e=1$ for all $e\in \E(\Ga)\smallsetminus \E(T)$. The two definitions coincide by \cite[Proposition 20, Chapitre I, \S5]{Serre_arbres}.
\end{defn}

A graph of groups is \textit{finite} if its underlying graph is finite. If $G$ is isomorphic to the fundamental group of a graph of groups, we say that $G$ \textit{splits} as a graph of groups. In this situation, we will often abuse terminology and say that $G$ \textit{is} a graph of groups. If the vertex and edge groups of a graph of groups $\mathcal G$ lie in classes $\mathcal C$ and $\mathcal D$, respectively, then we will say that $\mathcal G$ (or its fundamental group) is a graph of $\mathcal C$ groups with $\mathcal D$ edge groups. We will be mostly interested in graphs of free groups with cyclic edge groups in this article. A notable subclass which will appear is the class of \textit{generalised Baumslag--Solitar groups}, which are the groups that split as finite graphs of $\Z$'s with $\Z$ edge groups.

\begin{defn}\label{def:subgraph}
    Let $\mathcal G = (G_v, G_e; \Gamma)$ be a graph of groups. A graph of groups $\mathcal H = (H_v, H_e; \Upsilon)$ is a \textit{subgraph of groups} of $\mathcal G$ if 
    \begin{enumerate}
        \item there is an injection $\Upsilon \hookrightarrow \Gamma$ (via which we think of $\Upsilon$ as a subgraph of $\Gamma$),
        \item there are inclusions $f_v \colon H_v \longhookrightarrow G_v$ and $f_e \colon H_e \longhookrightarrow G_e$ for all vertices and edges of $\Upsilon$ (via which we think of every $H_v$ (resp.~$H_e$) as a subgroup of $G_v$ (resp.~$G_e$)), 
        \item $H_e = H_{\o(e)} \cap G_e$ and $H_e = H_{\t(e)} \cap G_e$ for every $e \in \E(\Upsilon)$, and
        \item the diagrams
        \[
            \begin{tikzcd}
                H_e \arrow[d, "f_e", hook] \arrow[r, hook] & H_{\mathsf o(e)} \arrow[d, "f_{\mathsf o(e)}", hook] \\
                G_e \arrow[r, hook]                        & G_{\mathsf o(e)}
            \end{tikzcd} \quad \text{and} \quad
            \begin{tikzcd}
                H_e \arrow[d, "f_e", hook] \arrow[r, hook] & H_{\mathsf t(e)} \arrow[d, "f_{\mathsf t(e)}", hook] \\
                G_e \arrow[r, hook]                        & G_{\mathsf t(e)}
            \end{tikzcd}
        \]
        commute for all $e \in \E(\Upsilon)$, where the horizontal maps are the edge maps of the respective graphs of groups.
    \end{enumerate}
\end{defn}

\begin{lem}[{\cite[Corollary 1.14]{Bas93}}]\label{lem:subgraph_inject}
    If $\mathcal G$ is a graph of groups and $\mathcal H$ is a subgraph of groups, then there is a canonical injective homomorphism $\pi_1(\mathcal H, v) \longhookrightarrow \pi_1(\mathcal G, v)$ for any vertex $v$ in the underlying graph of $\mathcal H$.
\end{lem}

Finally, note that if $H$ is an arbitrary subgroup of a graph of groups $G$, then $H$ inherits a graph of groups structure, which comes from the action of $H$ on the Bass--Serre tree associated to $G$ (see \cite[Th\'eor\`eme 13, Chapitre I, \S5]{Serre_arbres}).

We will often switch between the graph of groups and graph of spaces viewpoint, the latter of which we introduce now.

\begin{defn}[Graph of spaces]
    A \textit{graph of spaces} $\mathcal X$ consists of the following data:    
    \begin{enumerate}
        \item a connected directed graph $\Gamma$, called the \textit{underlying graph} of $\mathcal X$;
        \item based connected CW-complexes $(X_v,x_v)$ and $(X_e,x_e)$ for every vertex $v \in \V(\Gamma)$ and edge $e \in \E(\Gamma)$;
        \item based $\pi_1$-injective continuous maps $f_{e,\o} \colon X_e \lrar X_{\o(e)}$ and $f_{e,\t} \colon X_e \lrar X_{\t(e)}$ for every edge $e \in \E(\gamma)$.
    \end{enumerate}
    The spaces $X_v$ and $X_e$ are called the \textit{vertex spaces} and the \textit{edge spaces} of $\mathcal X$. The maps $f_{e,\o}$ and $f_{e,\t}$ are called the edge maps. The \textit{geometric realisation} of $\mathcal X$ is the quotient of the space
    \[
        X = \left(\bigsqcup_{v \in \V(\Gamma)} X_v\right)  \sqcup  \left( \bigsqcup_{e \in \E(\Gamma)} X_e \times [0,1] \right)
    \]
    by the relations $f_{e,\o}(x) \sim (x,0)$ and $\varphi_{e,1}(x) \sim (x,1)$ for all $x \in X_e$ and all $e \in \V(E)$. The \textit{fundamental group} of the topological space $\mathcal X$ based at $x_{v_0}$, denoted $\pi_1(\mathcal X,x_{v_0})$, is defined to be $\pi_1(X, x_{v_0})$. When no confusion arises, we will usually refer to the geometric realisation of $\mathcal X$ as a graph of spaces.
\end{defn}

There is a correspondence between graphs of spaces and graphs of groups. If $\mathcal G$ is a graph of groups, then a graph of spaces $\mathcal X$ can be constructed as follows: For each $v \in \V(\Gamma)$ and $e \in \E(\Gamma)$, let $X_v = K(G_v,1)$ and $X_e = K(G_e,1)$, and let $f_{e,\o}$ and $f_{e,\t}$ be maps inducing $\varphi_{e,\o}$ and $\varphi_{e,\t}$. Then there is an isomorphism between $\pi_1(\mathcal G, v_0)$ and $\pi_1(\mathcal X, x_{v_0})$ (which depends on the choices of $v_0$ and $x_{v_0}$ up to conjugation). Similarly, given a graph of spaces, we can form a graph of groups with vertex groups $\pi_1(X_v, x_v)$, edge groups $\pi_1(X_e, x_e)$, and edge maps $(f_{e,\o})_*$ and $(f_{e,\t})_*$.

One of our main results concerns graphs of groups where all the vertex groups are free groups and all the edge groups are infinite cyclic. Any such group will be referred to as a \emph{graph of free groups with cyclic edge groups}. Such groups are realised as the fundamental group of a \textit{graph of graphs with $S^1$ edge spaces}, by which we understand a graph of spaces where every vertex space is a graph and every edge space is a copy of the circle $S^1$.

The notion of a precovering will appear throughout \cref{sec:GofFG,sec:ICE}, so we recall it here. It shows up naturally when completing a compact subspace of a covering space to a finite-sheeted covering.

\begin{defn}
    A map between (the geometric realisations of) graphs of spaces $X'\lrar X$ is a {\it precovering} if it is locally injective, all the maps $X_e'\lrar X_{f(e)}$ and $X_v'\lrar X_{f(v)}$ are covering maps, and all the diagrams
    \[
        \begin{tikzcd}
            X_e' \arrow[r] \arrow[d] & X_{\mathsf{o}(e)}' \arrow[d] \\
            X_{f(e)} \arrow[r]       & X_{f(\mathsf o(e))}         
        \end{tikzcd} \quad \text{and} \quad
        \begin{tikzcd}
            X_e' \arrow[r] \arrow[d] & X_{\mathsf{t}(e)}' \arrow[d] \\
            X_{f(e)} \arrow[r]       & X_{f(\mathsf t(e))}         
        \end{tikzcd}
    \]
    commute. The domain $X'$ is called a {\it precover}.  
\end{defn}

A precovering $X'\lrar X$ is a covering if and only if all the elevations of edge maps of $X$ to $X'$ are edge maps of $X'$. The fact that covering maps induce injections on fundamental groups also applies to precoverings.

\begin{lem}[{\cite[Proposition 2.19]{WiltonHall}}]
    A precovering $X' \lrar X$ induces an injection $\pi_1(X') \lrar \pi_1(X)$.
\end{lem}

We will also require subgraphs of spaces, which induce subgraphs of groups in the sense of \cref{def:subgraph}.

\begin{defn}\label{def:subgraph_spaces}
    Let $\mathcal X = (X_v, X_e; \Gamma)$ be a graph of spaces. A graph of spaces $\mathcal Y = (Y_v, Y_e; \Upsilon)$ is a \textit{subgraph of spaces} of $\mathcal X$ if
    \begin{enumerate}
        \item there is an injection $\Upsilon \hookrightarrow \Gamma$ (via which we think of $\Upsilon$ as a subgraph of $\Gamma$),
        \item there are $\pi_1$-injective inclusions $f_v \colon Y_v \longhookrightarrow X_v$ and $f_e \colon Y_e \longhookrightarrow X_e$ for all edges and vertices of $\Upsilon$ (via which we think of every $Y_v$ (resp.~$Y_e$) as a subspace of $X_v$ (resp.~$X_e$)),
        \item $\pi_1(Y_e) = \pi_1(Y_{\o(e)}) \cap \pi_1(X_e)$ and $\pi_1(Y_e) = \pi_1(Y_{\t(e)}) \cap \pi_1(X_e)$ for every edge $e \in \E(\Upsilon)$, and
        \item the diagrams
        \[
            \begin{tikzcd}
                Y_e \arrow[d, "f_e", hook] \arrow[r, hook] & Y_{\mathsf o(e)} \arrow[d, "f_{\mathsf o(e)}", hook] \\
                X_e \arrow[r, hook]                        & X_{\mathsf o(e)}
            \end{tikzcd} \quad \text{and} \quad
            \begin{tikzcd}
                Y_e \arrow[d, "f_e", hook] \arrow[r, hook] & Y_{\mathsf t(e)} \arrow[d, "f_{\mathsf t(e)}", hook] \\
                X_e \arrow[r, hook]                        & X_{\mathsf t(e)}
            \end{tikzcd}
        \]
        commute for all $e \in \E(\Upsilon)$, where the horizontal maps are edge maps in the corresponding graphs of spaces.
    \end{enumerate}
\end{defn}

We close by remarking that if $X$ decomposes as a graph of spaces and $Y \lrar X$ is a covering space, then $Y$ inherits a graph of spaces structure where every vertex (resp.~edge) space of $Y$ covers some vertex (resp.~edge) space of $X$.

\subsection{Homology of groups}

Unless stated otherwise, all modules are assumed to be left modules. Let $R$ be a ring. Given a right $R$-module $M$ and a left $R$-module $N$, we can define the abelian group $\Tor_i^R(M, N)$. By definition, $\Tor_0^R(M, N)\cong M\otimes_R N$ as an abelian group. In general, the functors $\Tor_n^R(M,-)$ are the derived functors of $M \otimes_R -$. More concretely, we choose a projective resolution $P_\bullet \lrar N \lrar 0$ and define $\Tor_n^M(M,N) := H_n(M \otimes_R P_\bullet)$.

Let $S$ be another ring. If $M$ is additionally an $(S, R)$-bimodule, then $\Tor_n^R(M, N)$ is naturally a left $S$-module for all $n$. Similarly, if $N$ is an $(R, S)$-bimodule, then $\Tor_n^R(M, N)$ is naturally a right $S$-module. A standard tool we will use is the long exact sequence in $\Tor$ associated to a short exact sequence of modules. Let $0 \lrar N_1 \lrar N_2 \lrar N_3 \lrar 0$ be a short exact sequence of $R$-modules and let $M$ be an $(S, R)$-bimodule. Then there is a long exact sequence of left $S$-modules of the form
\[
    \begin{tikzcd}[row sep=small]
        & \cdots \arrow[r]\arrow[d, phantom, ""{coordinate, name=Z}]& \Tor_{n+1}^R(M,N_3) \arrow[dll,  rounded corners, to path={ -- ([xshift=2ex]\tikztostart.east)|- (Z) [near end]\tikztonodes-| ([xshift=-2ex]\tikztotarget.west)-- (\tikztotarget)}] &  \\
        \Tor_n^R(M,N_1) \arrow[r]& \Tor_n^R(M,N_2) \arrow[r] & \Tor_n^R(M,N_3) \arrow[r] & \cdots.
    \end{tikzcd}
\]
A standard reference for this material is \cite[Chapters 2 and 3]{Weibel_HA}.

Let $G$ be a group and let $M$ be an $R[G]$-module. As in \cite[Chapter III, Section 2]{BrownGroupCohomology}, the $n$-dimensional homology of $G$ with coefficients in $M$ is given by
\[
    H_n(G;M) := \Tor_n^{R[G]}(R,M)
\]
where $R$ denotes the trivial right $R[G]$-module. Chiswell's Mayer--Vietoris exact sequence will be a very useful tools when establishing the $L^2$-Hall property for certain graphs of groups.

\begin{thm}[{\cite[Theorem 2]{ChiswellMV1976}}]\label{thm:MV}
    Let $R$ be a ring, let $\mathcal G$ be a graph of groups with underlying graph $\Gamma$ and $G = \pi_1(\mathcal G)$, and let $M$ be an $R[G]$-module. Then there is a long exact sequence
    \[
        \begin{tikzcd}[sep=small]
            & \cdots \arrow[r]\arrow[d, phantom, ""{coordinate, name=Z}]& H_{n+1}(G;M) \arrow[dll,  rounded corners, to path={ -- ([xshift=2ex]\tikztostart.east)|- (Z) [near end]\tikztonodes-| ([xshift=-2ex]\tikztotarget.west)-- (\tikztotarget)}] &  \\
            \bigoplus_{e \in \E(\Gamma)} H_n(G_e;M) \arrow[r]& \bigoplus_{v \in \V(\Gamma)} H_n(G_v;M) \arrow[r] & H_n(G;M) \arrow[r] & \cdots.
        \end{tikzcd}
    \]
\end{thm}

Given a field $K$ and a group $G$, denote by $I_G$ the augmentation ideal of the group ring $K[G]$ (in practice, this notation will present no ambiguity as the choice of coefficient field $K$ will be clear from the context). Given a subgroup $H \leq G$, we denote by $I_H^G$ the left $K[G]$-submodule of $I_G$ generated by $I_H$. In addition, even if $H$ is not normal in $G$, we will write $K[G/H]$ to refer to the left $K[G]$-module of left cosets of $H$ in $G$. The following canonical isomorphisms will be useful later. 

\begin{lem}[\cite{Ja23}, Lemma 2.1] \label{augiso} Let $T\leq H\leq G$ be subgroups. Then the following hold.
\begin{enumerate}
    \item The canonical map $K[G]\otimes_{K[H]}I_H\lrar I_H^G$ that sends $a\otimes b$ to $a\cdot b$ for all $a\in K[G]$ and $b\in I_H$ is an isomorphism of left $K[G]$-modules. 
    \item The canonical map $K[G]\otimes_{K[H]} \left(I_H/I_T^H\right)\lrar I_H^G/I_T^G$ that sends $a\otimes (b+I_T^H)$ to $ab+I_T^G$ for all $a\in K[G]$ and $b\in I_H$ is an isomorphism of left $K[G]$-modules.
    \item The kernel of the canonical map of $K[G]$-modules $K[G/T]\lrar K[G/H]$ is naturally isomorphic to $I_H^G/I_T^G$.
\end{enumerate}
\end{lem}

\subsection{Hughes-free division rings and \texorpdfstring{$L^2$}{L²}-Betti numbers}

Let $G$ be a locally indicable group and let $K$ be a field. An embedding $\varphi \colon K[G] \longhookrightarrow \mathcal D$ of the group algebra $K[G]$ into a division ring $\mathcal D$ is called \textit{Hughes-free} if the following conditions hold.
\begin{enumerate}
    \item The image $\varphi(K[G])$ generates $\mathcal D$ as a division ring.
    \item Let $H \leqslant G$ be a finitely generated subgroup and let $f \colon H \lrar \Z$ be an epimorphism with kernel $N$, and let $t \in H$ map to a generator of $\Z$ under $f$. Let $\mathcal D_N$ denote the division closure of $\varphi(K[N])$. Then $\{ \varphi(t^i) : i \in \Z\} \subseteq \mathcal D$ is linearly independent over $\mathcal D_N$.
\end{enumerate}

By a theorem of Hughes, if a Hughes-free embedding of $K[G]$ exists, then it is unique up to $K[G]$-isomorphism \cite{HughesDivRings1970}. Thus, if $K[G]$ has a Hughes-free embedding, then we denote the division ring by $\Dg$ and think of $K[G]$ as a subset of $\Dg$. We will call $\D_{K[G]}$ the \textit{Hughes-free division ring} of $K[G]$. Note that if $H \leqslant G$ is any subgroup, then the division closure of $K[H]$ in $\D_{K[G]}$ is isomorphic to the Hughes-free division ring $\Dh$. The existence of Hughes-free division rings has been established for many classes of locally indicable groups, and in particular for all locally indicable groups when the ground field $K$ has characteristic zero.

\begin{prop}\label{prop:DKGexistence}
    Let $G$ be locally indicable. A Hughes-free embedding $K[G] \hookrightarrow \D_{K[G]}$ exists if
    \begin{enumerate}
        \item\label{item:char0} the field $K$ is of characteristic zero, or 
        \item\label{item:VCS} if $G$ residually (locally indicable and amenable) or virtually compact special.
    \end{enumerate}
\end{prop}
\begin{proof}
    If $K$ is of characteristic zero, then the existence of $\Dg$ is a consequence of the resolution of the Atiyah conjecture for locally indicable groups \cite[Corollary 1.4]{ZapirainLopezStrongAtiyah2020}. If $K$ is of arbitrary characteristic, then a Hughes-free embedding $K[G] \hookrightarrow \Dg$ exists for $G$ residually (locally indicable and amenable) by \cite[Corollary 1.3]{JaikinZapirain2020THEUO} and for $G$ virtually compact special by \cite[Theorem 1.2]{fishersanchez_divrings}. \qedhere
\end{proof}

The groups we will be working with in this article are locally indicable and virtually compact special, so we will always assume that any group algebra $K[G]$ has a Hughes-free division ring $\Dg$.

\begin{lem}
    Let $G$ be a finite graph of finitely generated free groups with cyclic edge groups. If $K$ is a field of characteristic zero, then a Hughes-free embedding $K[G] \hookrightarrow \Dg$ exists. If we assume that $G$ is subgroup separable, then a Hughes-free embedding $K[G] \hookrightarrow \Dg$ exists for arbitrary $K$.
\end{lem}
\begin{proof}
    First note that $G$ is locally indicable, a fact which follows easily from \cite[Theorem 4.2]{Howie_locIndgps}. Thus if $K$ is of characteristic zero, then a Hughes-free embedding $K[G] \hookrightarrow \Dg$ exists by \cref{prop:DKGexistence}(1). If $G$ is subgroup separable, then $G$ is virtually compact special \cite[Corollary 2.3]{MinasyanMineh_QC} and thus a Hughes-free embedding $K[G] \hookrightarrow \Dg$ exists by \cref{prop:DKGexistence}(2). \qedhere
\end{proof}

\begin{rem}
    Let $G$ be a graph of free groups with cyclic edge groups. It is known that $K[G]$ embeds in a division ring by \cite[Theorem 1.3]{fishersanchez_divrings}, but it is not known whether the embedding is Hughes-free. Jaikin-Zapirain conjectures that Hughes-free embeddings of $K[G]$ exist for all locally indicable groups $G$ and all fields $K$ \cite[Conjecture 1]{JaikinZapirain2020THEUO}.
\end{rem}

Hughes-free division rings provide powerful homological invariants. Recall that modules over a division ring are automatically free modules and that they have a well-defined dimension. Thus, if $M$ is a $K[G]$-module, we can define its $\D_{K[G]}$-dimension by 
\[
    \dim_{\D_{K[G]}} M := \dim_{\D_{K[G]}} (\D_{K[G]} \otimes_{K[G]} M)
\]
and more generally $\Dg$-Betti numbers by
\begin{equation} \label{DgBettieq2}
    \beta_n^{K[G]}(M) := \dim_{\D_{K[G]}} \Tor_n^{K[G]}(\D_{K[G]}, M).
\end{equation}
We will not need these $\Dg$-Betti numbers of general $K[G]$-modules until \cref{sec:L2SHNC}.
Note that $\beta_0^{k[G]}(M) = \dim_{\D_{K[G]}} M$. When $K = \C$, we will write $\beta_n^{(2)}(M)$ instead of $\beta_n^{K[G]}(M)$. Setting $K$ to be the trivial $K[G]$-module, we obtain homological numerical invariants of the group $G$:
\begin{equation} \label{DgBettieq}
    b_n^{K[G]}(G) := \beta_n^{K[G]}(K) = \dim_{\D_{K[G]}}  \Tor_n^{K[G]}(\D_{K[G]}, K).
\end{equation}
We will refer to these as the $\D_{K[G]}$-Betti numbers of $G$.

The properties listed in the following proposition will be used throughout the article. We emphasize point (1) below, which states that when $K = \C$, the $\Dg$-Betti numbers coincide with the $L^2$-Betti numbers of $G$.

\begin{prop}\label{prop:HFprops}
    Let $G$ be a  locally indicable group and let $K$ be a field such that a Hughes-free embedding $K[G] \hookrightarrow \D_{K[G]}$ exists.
    \begin{enumerate}[label = (\arabic*)]
        \item\label{item:K=C} If $K= \C$, then $b_n^{K[G]}(G) = \b_n(G)$ for all integers $n\geq 0$.
        \item\label{item:b0} If $G$ is nontrivial, then $b_0^{K[G]}(G) = 0$, otherwise $b_0^{K[G]}(G) = 1$.
        \item\label{item:euler} If $G$ is of finite type, then $\chi(G) = \sum_{i=0}^\infty (-1)^i b_i^{K[G]}(G)$.
        \item\label{item:fibetti} Let $H \leqslant G$ be a subgroup of finite index. Then $\Dh\otimes_{K[H]} K[G] \cong \Dg$ as $(\Dh, K[G])$-bimodules. Consequently,  $b_n^{K[H]}(H) = |G:H| \cdot b_n^{K[G]}(G)$ for all $n$.
        
        \item \label{item:Shapi0} Let $k\geq 0$ be an integer, let $H$ be a subgroup of $G$ and let $N$ be a left $K[H]$-module. There is an isomorphism of left $\D_{K[G]}$-modules of the form
        \[\Tor_k^{K[H]}(\D_{K[G]}, N)\cong \Tor_k^{K[G]}(\D_{K[G]}, K[G]\otimes_{K[H] }N).\]
        In particular,  if $N=K$ is the trivial $K[H]$-module, then 
        \[b_k^{K[H]}(H)=\dim_{\D_{K[G]}} \Tor_k^{K[G]} (\D_{K[G]}, K[G/H]).\]
        \item \label{item:Shapi1} Let $k\geq 0$ be an integer, let   $H$ be a finite-index subgroup of $G$ and let $M$ be a left $K[G]$-module. There is an isomorphism of left $\D_{K[H]}$-modules of the form
        \[\Tor_k^{K[H]}(\D_{K[H]}, M)\cong \Tor_k^{K[G]}(\D_{K[G]}, M).\]
        \item\label{item:freebetti} If $G$ is free on $m \geqslant 1$ generators, then $b_1^{K[G]}(G) = m-1$ and $b_n^{K[G]}(G) = 0$ for all $n \neq 1$. If $G$ is amenable, then $b_n^{K[G]}(G) = 0$ for all $n$.
    \end{enumerate}
\end{prop}

\begin{proof}
    Statement \ref{item:K=C} follows from \cite[Theorem 1.1]{ZapirainLopezStrongAtiyah2020}, while \ref{item:b0} can easily be proven directly from the definitions. To prove \ref{item:euler}, let
    \[
        0 \lrar K[G]^{r_d} \lrar K[G]^{r_{d-1}} \lrar \dots \lrar K[G]^{r_0} \lrar K \lrar 0
    \]
    be a resolution of the trivial \(K[G]\)-module \(K\) by finitely generated free \(K[G]\)-modules. By definition, \(\chi(G) = \sum_{i=0}^d (-1)^i r_i\). After tensoring with \(\Dg\) and omitting the rightmost term, we obtain the chain complex
    \[
        0 \lrar \Dg^{r_d} \lrar \Dg^{r_{d-1}} \lrar \dots \lrar \Dg^{r_0} \lrar 0
    \]
    whose boundary maps we denote by \(\partial_i \colon \Dg^{r_i} \lrar \Dg^{r_{i-1}}\). Since \(\Dg\) is a division ring, the rank-nullity theorem holds, and therefore there is a decomposition
    \[
        \Dg^{r_i} \cong \ker \partial_i \oplus \im \partial_i \cong \Tor_i^{K[G]}(\Dg,K) \oplus \im \partial_{i+1} \oplus \im \partial_i.
    \]
    Since, by definition, \(b_i^{K[G]}(G) = \dim_{\Dg}\Tor_i^{K[G]}(\Dg,K)\), we obtain \(\chi(G) = \sum_{i=0}^\infty (-1)^i b_i^{K[G]}(G)\).
    
    Statement \ref{item:fibetti} is a direct consequence of \cite[Corollary 8.3]{Grater20} (for a detailed proof see \cite[Lemma 6.3]{Fisher2021improved}). The isomorphism of  \ref{item:Shapi0} follows from an standard application of Shapiro's lemma on the second entry  of the $\Tor$ functor. The second equation of \ref{item:Shapi1} follows from setting the trivial $K[H]$-module $N=K$ and from noting that, as left $K[G]$-modules, $K[G]\otimes_{K[H]}K\cong K[G/H]$. Similarly, we apply Shapiro's lemma to the first entry of $\Tor$ to obtain  the isomorphism 
    \[ \Tor_k^{K[H]}(\D_{K[H]}, M)\cong \Tor_k^{K[G]}(\D_{K[H]}\otimes_{K[H]} K[G], M). \]
    Now \ref{item:Shapi1} follows from  \ref{item:fibetti} and the previous isomorphism. 
    Finally, for \ref{item:freebetti}, the claim about free groups can be proved easily using \ref{item:b0} and \ref{item:euler}. The claim about amenable groups follows from \cite[Theorem 3.9(6)]{HennekeKielak2021} (only the case $K = \Q$ is treated there, but the case with $K$ arbitrary has the same proof). \qedhere
\end{proof}

\section{\texorpdfstring{$L^2$}{L²}-independence and the \texorpdfstring{$L^2$}{L²}-Hall property} \label{sec:L2hall}

In this section we discuss the $L^2$-Hall property and the concept of $L^2$-independent subgroups in more detail. We then study various combinatorial situations (in terms of graphs of groups) that provide $L^2$-independent subgroups (which we shall need in the proofs of Theorems \ref{thm:A} and \ref{thm:B}) and show in \cref{overgroups} that the $L^2$-Hall property passes to finite-index overgroups in our setting (as anticipated in \cref{thm:D}). 

\begin{conv}\label{conv:sec3conv}
    In this section, $K$ always denotes a field. Apart from in some isolated examples, all groups appearing are assumed to be locally indicable and we assume that their group algebras over $K$ have Hughes-free embeddings (recall that this is the case when $\mathrm{char}\, K=0$ by \cref{prop:DKGexistence}).
\end{conv}

\subsection{Definitions and basic properties}

The notion of $L^2$-independence and the $L^2$-Hall property were introduced Antolín--Jaikin-Zapirain \cite{Jai22} in connection with proving that surface groups satisfy the Strengthened Hanna Neumann Conjecture. We recall these definitions.

\begin{defn}\label{def:L2indep}
    Let $H$ be a subgroup of $G$. Consider the natural surjection of left $K[G]$-modules $K[G/H]\lrar K$. This induces a natural map 
    \[
        \Tor_1^{K[G]}(\Dg, K[G/H]) \lrar \Tor_1^{K[G]}(\Dg, K).
    \]
    We say that $H$ is {\it $\Dg$-independent} if the map is injective. When $K = \C$, we will say that $H$ is \textit{$L^2$-independent} in $G$. 
\end{defn}

The injectivity of the above map depends on the choice of embedding of $H$ into $G$. For example, the embedding $f\po F(a, b, c)\lrar G= F(x, y, z)$ defined by $f(a)=x^2$, $f(b)=y$ and $f(c)=y^x$ does not lead to an $L^2$-independent subgroup of $G$. For this reason, the following definition will be useful later. 

\begin{defn} \label{def:Dginjective}
    Given a monomorphism $f\po H \longhookrightarrow G$, we will say that $f$ is {\it $\Dg$-injective} if $f(H)$ is $\Dg$-independent in $G$ (or {\it $L^2$-injective} when $K=\C).$
\end{defn}

By \cite[Proposition 4.2]{Jai22}, $H$ is $\Dg$-independent in $G$ if and only if the co-restriction map $H_1(H;\Dg) \lrar H_1(G;\Dg)$ is injective. So \cref{def:L2indep} is the natural generalisation of Antolín--Jaikin-Zapirain's definition of $L^2$-independence  \cite[Section 4]{Jai22} for other division rings $\Dg$. Working in this greater generality will uniformly include various cases of interest while adding no technical difficulty.

The augmentation ideal corresponding to a subgroup captures a lot of structure of the subgroup and hence \cref{prop:L2indep2} provides a useful reformulation of the notion of $\Dg$-independence.

\begin{prop}[{\cite[Corollary 4.3]{Jai22}}]\label{prop:L2indep2} 
    Let $H\leq U\leq G$ be finitely generated subgroups and suppose that $\bG_2(G)=0$. Then $H$ is $\Dg$-independent in $U$ if and only if $\bG_1(I_U^G/I_H^G)=0$. 
\end{prop}

\begin{defn}\label{def:L2Hall}
    We say that a group $G$ is {\it $\Dg$-Hall} or has the {\it $\Dg$-Hall property} if for every finitely generated subgroup $H\leqslant G$ there exists a finite-index subgroup $G_1\leqslant G$ such that $H$ is $\Dg$-independent in $G_1$. If $K = \C$, we say that $G$ is \textit{$L^2$-Hall} or has the \textit{$L^2$-Hall property}.
\end{defn}

\begin{rem}
    Note that the $L^2$-Hall property can be defined for all groups, while the $\Dg$-Hall property only makes sense for locally indicable groups for which a Hughes-free embedding $K[G] \hookrightarrow \Dg$ exists. Indeed, if $H \leqslant G$ and $\mathcal U(G)$ is the algebra of affiliated operators of $G$, then we say that $H$ is \textit{$L^2$-independent} in $G$ if 
    \begin{equation}\label{eq:L2_hall}
        \dim_{\mathcal U(G)} \ker\left(H_1(H; \mathcal U(G)) \lrar H_1(G; \mathcal U(G))\right) = 0
    \end{equation}
    and that $G$ has the \textit{$L^2$-Hall property} if every finitely generated subgroup of $G$ is $L^2$-independent in a finite-index subgroup of $G$. These definitions agree with \cref{def:L2indep,def:L2Hall} by \cite[Lemma 4.1]{Jai22}. We mention this because we will discuss the $L^2$-Hall property for some non locally indicable groups later in this section. On the other hand, an advantage of working with the $\Dg$-Hall property is that the condition that $H_1(H;\Dg) \lrar H_1(G;\Dg)$ be injective is somewhat less awkward than the condition in \ref{eq:L2_hall}.
\end{rem}

The following hereditary feature of the $L^2$-Hall property will be useful later. Recall that a ring homomorphism \(R \lrar S\) is \emph{(right) faithfully flat} if for every morphism \(M \lrar N\) of (right) \(R\)-modules, \(M \lrar N\) is injective if and only if \(M \otimes_R S \lrar N \otimes_R S\) is injective. There is the corresponding concept of left faithful flatness, which is defined analogously. If \(\mathcal D_1 \lrar \mathcal D_2\) is a morphism of division rings, then it is necessarily injective and (left and right) faithfully flat. Indeed, consider a morphism of \(\mathcal D_1\)-modules \(M \lrar N\). Since \(\mathcal D_1\) is a division ring, \(\mathcal D_2 \cong \oplus_I \mathcal D_1\) for some index set \(I\). From the commutative diagram
\[
    \begin{tikzcd}
        M \otimes_{\mathcal D_1} \mathcal D_2 \arrow[r] \arrow[d, "\cong", no head] & N \otimes_{\mathcal D_1} \mathcal D_2 \arrow[d, "\cong", no head] \\
        \bigoplus_I M \arrow[r] & \bigoplus_I N \nospacepunct{,}
    \end{tikzcd}
\]
it follows at once that \(M \lrar N\) is injective if and only if \(M \otimes_{\mathcal D_1} \mathcal D_2 \lrar N \otimes_{\mathcal D_1} \mathcal D_2\) is.

\begin{lem}\label{lem:l2subgp}
    The $\Dg$-Hall property passes to subgroups.
\end{lem}

\begin{proof}
    Let $G$ be a $\Dg$-Hall group and let $H \leqslant G$ be a subgroup. Let $U \leqslant H$ be a finitely generated subgroup. Then there is a subgroup $G_0 \leqslant G$ of finite index such that the horizontal map in the diagram 
    \[
        \begin{tikzcd}[row sep = small, column sep = tiny]
            H_1(U;\D_{K[G_0]}) \arrow[rr] \arrow[rd] & & H_1(G_0;\D_{K[G_0]}) \\
            & H_1(G_0 \cap H;\D_{K[G_0]}) \arrow[ru] &
        \end{tikzcd}
    \]
    is injective. But then $H_1(U;\D_{K[G_0]}) \lrar H_1(G_0 \cap H; \D_{K[G_0]})$ is injective. Since extensions of division rings are faithfully flat, the commutative diagram 
    \[
        \begin{tikzcd}[column sep = small]
            H_1(U; \D_{K[G_0]}) \arrow[d, "\cong", no head] \arrow[r]    & H_1(G_0 \cap H; \D_{K[G_0]}) \arrow[d, "\cong", no head] \\
            \mathcal \D_{K[G_0]} \underset{\D_{K[G_0 \cap H]}}{\otimes} H_1(U; \D_{K[G_0 \cap H]})  \arrow[r] &  \mathcal \D_{K[G_0]} \underset{\D_{K[G_0 \cap H]}}{\otimes} H_1(G_0 \cap H; \D_{K[G_0 \cap H]})     
        \end{tikzcd}
    \]
    implies that $H_1(U; \D_{k[G_0 \cap H]}) \longrightarrow H_1(G_0 \cap H; \D_{k[G_0 \cap H]})$ is injective. Hence, $H$ has the $L^2$-Hall property. \qedhere
\end{proof}

We now collect various instances where we understand $L^2$-independent subgraphs of groups. The first of such examples is \cref{lem:MV} and will be useful when establishing the $L^2$-Hall property for graphs of free groups with cyclic edge groups.

\begin{lem} \label{lem:MV}
    Let $\mathcal Y$ be a subgraph of groups of $\mathcal Z$, and let $Z := \pi_1(\mathcal Z)$. If
    \begin{enumerate}
        \item the maps $Y_v\lrar Z_v$ are $\D_{K[Z]}$-injective for all $v\in \V(\Ga^\mathcal Y)$,
        \item $b_1^{\D_{K[Z]}}(Z_e) = 0$ for all $e\in \E(\Ga^\mathcal Z)$, and
        \item the groups $Y_e$ and $Z_e$ are isomorphic for all $e\in \E(\Ga^\mathcal Y)$
    \end{enumerate}
    then the canonical injection $\pi_1(\mathcal Y) \lrar \pi_1(\mathcal Z)$ is $\D_{K[Z]}$-injective. 
\end{lem}

\begin{proof}
    We view $\pi_1(\mathcal Y)$ as a subgroup of $Z$ via the canonical inclusion. The subgraph of groups $\mathcal Y$ of $\mathcal Z$ induces a map of exact sequences
    \[
        \begin{tikzcd}
            0 \arrow[r] & {\underset{e \in \E(\Gamma^\mathcal Y)}{\bigoplus} K[Y/Y_v]} \arrow[r] \arrow[d] & {\underset{v \in \V(\Gamma^\mathcal Y)}{\bigoplus} K[Y/Y_v]} \arrow[r] \arrow[d] & K \arrow[r] \arrow[d] & 0 \\
            0 \arrow[r] & {\underset{e \in \E(\Gamma^\mathcal Z)}{\bigoplus} K[Z/Z_e]} \arrow[r]           & {\underset{v \in \E(\Gamma^\mathcal Z)}{\bigoplus} K[Z/Z_v]} \arrow[r]           & K \arrow[r]           & 0 \nospacepunct{.}
        \end{tikzcd}
    \]
    Since Chiswell's Mayer--Vietoris exact sequence is induced by applying a Tor functor to the short exact sequences of the above form (see the proof of \cite[Theorem 2]{ChiswellMV1976}), the long exact sequences are automatically natural and thus we obtain maps between Chiswell's exact sequences for $\mathcal Y$ and $\mathcal Z$:
    \[
        \begin{tikzcd}[column sep = small]
            \underset{e \in \E(\Gamma^\mathcal Y)}{\bigoplus} H_1(Y_e) \arrow[r] \arrow[d, "\cong"] & \underset{v \in \V(\Gamma^\mathcal Y)}{\bigoplus} H_1(Y_v) \arrow[r] \arrow[d, hook] & H_1(\pi_1(\mathcal Y)) \arrow[d] \arrow[r] & \underset{ve \in \E(\Gamma^\mathcal Y)}{\bigoplus} H_0(Y_e) \arrow[d, "\cong"] \\
            \underset{e \in \E(\Gamma^\mathcal Z)}{\bigoplus} H_1(Z_e) \arrow[r]                    & \underset{v \in \V(\Gamma^\mathcal Z)}{\bigoplus} H_1(Z_v) \arrow[r]                 & H_1(Z) \arrow[r]                           & \underset{e \in \E(\Gamma^\mathcal Z)}{\bigoplus} H_0(Z_e)   \nospacepunct{,}         
        \end{tikzcd}
    \]
    where $H_i(-)$ stands for $H_i(-;\mathcal D_{K[Z]})$ (for $i = 0,1$). By the Four Lemma, the map $H_1(\pi_1(\mathcal Y)) \lrar H_1(Z)$ is injective. \qedhere
\end{proof}

Our following technical proposition will be crucial to establish the $L^2$-Hall property for limit groups in \cref{sec:ICE}. We also consider it to be of potential interest for proving that relatively hyperbolic groups with a finite abelian hierarchy have the $L^2$-Hall property (\cref{conj:specialyHypSHNC}). 

\begin{prop} \label{prop:monster} 
    Let  $\mathcal W$ be a subgraph of groups of $\mathcal Z$ that have the same underlying graph $\Ga$ and all of whose edge groups are infinite cyclic. Let $G = \pi_1(\mathcal Z)$ and suppose that there is a bipartite structure $\V(\Ga) = \V_\o \sqcup \V_\t$ of $\Ga$ so that no two different edges of $\E(\Ga)$ have the same endpoints. We moreover assume that the orientation on $\Gamma$ is such that $\o(e)\in \V_{\o}$ and $\t(e)\in \V_{\t}$ for all $e\in \E(\Ga)$. We denote by $z_{e}$ a generator of the infinite cyclic group $Z_e$. Let $T$ be a spanning tree of $\Gamma$. Fix a presentation of $G$ (as described in \cref{defn:tree_pi1}) and, for every $e\in \E(\Ga)\smallsetminus \E(T)$, denote by $t_e$ the formal letter associated to $e$.  For all $v\in \V_{\o}$, we consider   finite subsets $\mathcal L_v^{(0)}\subset\mathcal L_v\subseteq Z_v$   such that
    \[
       \mathcal{L}_v \setminus \mathcal{L}_v^{(0)}\subseteq \bigcup_{\o(e)=v} \phi_{\o, e}(z_e) \subseteq W_v\cup \left(\mathcal{L}_v \setminus \mathcal{L}_v^{(0)}\right).
    \]
     Suppose that, for all $v\in \V_{\o}$,  the natural map 
    \[
        W_v*\left(\coprod_{\mathcal L_v} \Z \right)\lrar Z_v
    \]
    is injective and $\Dg$-injective. We fix a subset $  \E(T)\subseteq E_T\subseteq \E(\Ga)$ such that $\phi_{e, \o}(z_e)\in \mathcal L_{\o(e)}\setminus \mathcal L^{(0)}_{\o(e)}$ for all $e\in \E(\Gamma)\setminus E_T$. If we name $ \mathcal L^{(0)}=\bigcup_{v\in \V(\Ga)} \mathcal L_v^{(0)}$ and $\mathcal L^{(t)}=\{t_e:  e\in \E(\Gamma)\setminus E_T\}$, then  the natural map
    \begin{equation} \label{eq:W}
        \pi_1(\mathcal W)* \left(\coprod_{\mathcal L^{(0)}\bigcup \mathcal L^{(t)}} \Z\right)\lrar \pi_1 (\mathcal Z)
     \end{equation}
    is injective and $\Dg$-injective (in the sense of \cref{def:Dginjective}). 
\end{prop}

\begin{proof} 
    Recall from \cref{lem:subgraph_inject} that, given a subgraph of groups $\mathcal H$ of $\mathcal G$, the canonical map $\pi_1(\mathcal H)\lrar \pi_1(\mathcal G)$   injective. We will consider several intermediate graphs of groups $\mathcal W \leq \mathcal W^{(1)} \leq \mathcal W^{(2)} \leq \mathcal W^{(3)} \leq \mathcal W^{(4)} \leq \mathcal Z$ to prove the claim. We will only specify their vertex and edge groups, and the corresponding edge maps will be assumed to be the restrictions of the edge maps of $\mathcal Z$. 
    
    The graph of groups $\mathcal W^{(1)}$ is defined as follows. For all $v\in \V_{\t}$, $W_v^{(1)}=Z_v$. For all $v\in \V_{\t}$ and $e\in \E(\Ga)$, $W_v^{(1)}=W_v$ and $W_e^{(1)}=W_e$. By \cref{lem:MV}, the canonical map \begin{equation}  \label{eqW0} \pi_1(\mathcal W)\lrar \pi_1(\mathcal W^{(1)})\end{equation} is  $\Dg$-injective. 

    We split $\mathcal L_v\setminus \mathcal L_v^{(0)}$ as a disjoint union of $\mathcal L_v^{(1)}$ and $\mathcal L_v^{(2)}$, where $\mathcal L_v^{(1)}$ consists exactly of the elements $\phi_{\o, e}(z_e)\in \mathcal L_v$ such that $e\in  E_T$. Consider another intermediate graph of groups $\mathcal W^{(1)} \leq \mathcal W^{(2)} \leq \mathcal Z$ defined as follows:
    \begin{itemize}
        \item $W_v^{(2)}=W_v^{(1)}*\left(\coprod_{\mathcal{L}_v^{(1)}} \Z\right)$ for $v\in \V_{\o}$;
        \item $W_v^{(2)}=W_v^{(1)}$ for $v\in \V_{\t}$;
        \item $W_e^{(2)}=Z_e$ for $e\in   E_T$;
        \item $W_e^{(2)}=W_e^{(1)}$ for $e\in \E(\Ga)\smallsetminus E_T$.
    \end{itemize} 
    Letting $E^{(1)}=\{t_e : e\in E_T\smallsetminus E(T), \phi_{e, \o}(z_e)\in \mathcal L_v^{(1)}\}$, the canonical map 
    \begin{equation}\label{eq:W1}
        \pi_1(\mathcal W^{(1)})*\left( \coprod_{E^{(1)}} \Z \right) \lrar \pi_1(\mathcal W^{(2)})
    \end{equation}
    is an isomorphism, so $\pi_1(\mathcal W^{(1)})\lrar \pi_1(\mathcal W^{(2)})$ is $\Dg$-injective. 
    
    We define $\mathcal W^{(2)} \leq \mathcal W^{(3)}\leq \mathcal Z$ as follows:
    \begin{itemize}
        \item $W_v^{(3)}=W_v^{(2)}*\left(\coprod_{\mathcal{L}_v^{(0)}} \Z\right)$ for $v\in \V_{\o}$;
        \item $W_v^{(3)}=W_v^{(2)}$ for $v\in \V_{\t}$;
        \item $W_e^{(3)}=W_v^{(2)}$ for $e\in   E_T$;
        \item $W_e^{(3)}= W_v^{(2)}$ for $e\in \E(\Ga)\smallsetminus E_T$.  
    \end{itemize} 
    It is immediate to see  from the presentation of $\pi_1(\mathcal W^{(3)})$ that the canonical map 
    \begin{equation}\label{eq:W2}
        \pi_1(\mathcal W^{(2)})*\left(\coprod_{\mathcal L^{(0)}} \Z\right)\lrar \pi_1(\mathcal W^{(3)})
    \end{equation}
    is an isomorphism.  Finally, we define $\mathcal W^{(3)} \leq \mathcal W^{(4)} \leq \mathcal Z$ as follows:
    \begin{itemize} 
        \item $W_v^{(4)}=W_v^{(3)}*\left(\coprod_{\mathcal{L}_v^{(2)}} \Z\right)$ for $v\in \V_{\o}$;
        \item $W_v^{(4)}=W_v^{(3)}$ for $v\in \V_{\t}$; 
        \item $W_e^{(4)}=Z_e$ for $e\in \E(\Ga)$. 
    \end{itemize} 
    We observe that 
    \begin{equation}\label{eq:W3}
        \pi_1(\mathcal W^{(3)})*\left(\coprod_{\mathcal L^{(t)}} \Z\right)\lrar \pi_1(\mathcal W^{(4)})
    \end{equation}
    is an isomorphism. Observe that $ \mathcal W^{(4)} \leq \mathcal Z$ admits the following description: 
    \begin{itemize} 
        \item $W_v^{(4)}=W_v *\left(\coprod_{\mathcal{L}_v} \Z\right)$ for $v\in \V_{\o}$;
        \item $W_v^{(4)}=Z_v$ for $v\in \V_{\t}$; 
        \item $W_e^{(4)}=Z_e$ for $e\in \E(\Ga)$. 
    \end{itemize}  By our assumption on (\ref{eq:W}) and by \cref{lem:MV}, the canonical map 
    \begin{equation}\label{eq:W4}
        \pi_1(\mathcal W^{(4)})\lrar \pi_1(\mathcal Z)
    \end{equation}
    is $\Dg$-injective. From the chain of injections and $\Dg$-injections described in \eqref{eqW0}, \eqref{eq:W1}, \eqref{eq:W2}, \eqref{eq:W3} and \eqref{eq:W4}; we conclude that the canonical map
    \[
        \pi_1(\mathcal W)* \left(\coprod_{\mathcal L^{(0)} \bigcup \mathcal L^{(t)}} \Z\right)\lrar \pi_1 (\mathcal Z)
    \]
    is injective and $\Dg$-injective. The proof is complete.
\end{proof}

\subsection{Examples}

We are already in a position to establish the $\Dg$-Hall property for some classes of groups.

\begin{ex}[Amenable groups]
    Let $G$ be a group with the property that $\bG_1(H) = 0$ for all subgroups $H \leqslant G$. Then $G$ is trivially $\Dg$-Hall. Since amenable groups have vanishing $L^2$-Betti numbers above degree $0$ and amenability passes to subgroups, this shows that amenable groups are $L^2$-Hall. If $G$ is amenable and $K[G]$ is a domain (which is the case for us, since we are assuming \cref{conv:sec3conv}), then the same reasoning shows that $G$ is $\Dg$-Hall.

    There are also non-amenable groups which are $L^2$-Hall for the reason discussed above. As an example, let $T$ be a Tarski monster of prime order $p$ and let $G = T \times \Z$. Since all the proper subgroups of $T$ are isomorphic to $\Z/p$, it follows that every finitely generated subgroup $H$ of $G$ has $\b_1(H) = 0$ and therefore $G$ is $L^2$-Hall. However, $T$ is non-amenable, and therefore so is $G$. Note that $G$ is not locally indicable (or even torsion-free) and therefore $\Dg$ does not exist. However, it still makes sense to discuss the $L^2$-Hall property for this group since $L^2$-invariants are defined for all groups.
\end{ex}

\begin{ex}[Free groups] \label{freeHall}
    Let $F$ be a finitely generated free group and let $H \leqslant F$ be a finitely generated subgroup. A classical theorem of Marshall Hall \cite{Hal49} states that $H$ is a free factor in some finite-index subgroup $F' \leqslant F$. By \cref{lem:MV}, $H$ is $\mathcal{D}_{K[F]}$-independent in $F'$, showing that $F$ is $\mathcal{D}_{K[F]}$-Hall.
\end{ex}

Fundamental groups of closed surfaces also satisfy an analogous principle to Hall's theorem, namely that finitely generated subgroups are virtual retracts, as proved by Scott \cite{Scott_surfLerf} using hyperbolic geometry (see also \cite{WiltonPrehall} for a more combinatorial proof). This directly implies that surface groups are subgroup separable. Moreover, Antolín--Jaikin-Zapirain proved that they are $L^2$-Hall in \cite[Theorem 4.4]{Jai22}  using these virtual retractions combined with other algebraic ideas (such as the theory of Demushkin groups and the cohomological goodness of surface groups). We now use Scott's argument to give a more topological proof of the $L^2$-Hall property for surface groups.  

\begin{prop}\label{prop:surfaceHall}
    Surface groups satisfy the $\Dg$-Hall property.
\end{prop}

\begin{rem}
    The only surface that has a fundamental group with torsion is $\R P^2$, where $G = \pi_1(\R P^2) \cong \Z/2$. In this case, we may \textit{define} $b_0^{K[G]}(G) = \frac{1}{2}$ and $\bG_n(G) = 0$ for all $n \geqslant 1$, which is consistent with the index scaling formula \cref{prop:HFprops}\ref{item:fibetti}. In this sense, $\Z/2$ also has the $\Dg$-Hall property (and in fact so do all finite groups).
\end{rem}

\begin{proof}[Proof of \cref{prop:surfaceHall}]
    We say that a compact connected subsurface $X$ of a connected surface $S$ is {\it incompressible} if no component of the closure of the complement $S\smallsetminus X$ is a disc. If $\pi_1(X) \neq 1$, then $X$ is incompressible if and only if the induced map $\pi_1(X)\lrar \pi_1(S)$ is injective.
    
    Let $G$ be the fundamental group of a closed connected surface $\Sigma$ with $\chi(\Sigma)\leq 0$ (the case when $\chi(\Sigma)>0$ is trivial). Let $H \leqslant G$ be a non-trivial finitely generated subgroup. Let $\Sigma' \lrar \Sigma$ be the covering space corresponding to $H$. Then $\Sigma'$ is a (possibly non-compact) surface with fundamental group $H$. Let $\Sigma_c $ be a compact core for $\Sigma'$, that is,   $\Sigma_c \subseteq \Sigma'$ is a compact, connected, incompressible subsurface such that the natural map $\pi_1(\Sigma_c)\lrar \pi_1(\Sigma')$ is an isomorphism. The existence of $\Sigma_c$ is ensured by \cite[Lemma 1.5]{Scott_surfLerf}. Scott also showed in \cite[Lemma 1.4 and Theorem 3.3]{Scott_surfLerf} that there is a commutative diagram
    \[
        \begin{tikzcd}[column sep = tiny]
            & \widehat \Sigma \arrow[rd] &        \\
            \Sigma_c \arrow[ru, hook] \arrow[rr] &   & \Sigma
        \end{tikzcd}
    \]
    where $\widehat \Sigma \lrar \Sigma$ is an intermediate finite-sheeted covering into which $\Sigma_c$ projects homeomorphically.  Since $ \pi_1 (\Sigma_c)\cong H\neq 1$, the boundary  $\partial \Sigma_c$ is incompressible in $  \Sigma_c$. Consequently, $\Sigma_c$ is an incompressible subsurface of $\widehat \Sigma$ and every connected component $\widehat \Sigma_i$ of the closure of the complement  $\widehat\Sigma \smallsetminus \Sigma_c$ has the property that its boundary is incompressible. It follows that $\widehat \Sigma$ admits a decomposition as a finite graph of spaces where the vertex spaces are $\{\widehat \Sigma_i, \Sigma_c\}$, various of which are glued along some of their boundary components (so the edge spaces are circles and the edge maps are $\pi_1$-injective).  This produces a  splitting for the fundamental group  $\pi_1 (\widehat \Sigma)$ where one vertex is $\pi_1(\Sigma_c)$,  the other vertices are $\pi_1(\widehat \Sigma_i)$ and the edge groups are infinite cyclic.  By \cref{lem:MV}, the group $H=\pi_1(\Sigma_c)$ is $\Dg$-independent in $\pi_1(\widehat \Sigma)$, and therefore $G$ is $\Dg$-Hall. 
\end{proof}

The ideas of \cref{prop:surfaceHall}, such as the construction of a compact core for a subgroup $H$ and the reconstruction of $H$ from cyclic splittings,  motivates the strategy that we follow in \cref{thm:GOFGWCEG} for more general graphs of free groups with cyclic edges. We can now explain the simpler case when the edge groups are trivial (i.e. the case of free products), which generalises the proof that free groups are $\Dg$-Hall.

\begin{prop}\label{prop:L2freeProd}
    The class of finitely generated subgroup separable $\Dg$-Hall groups is closed under free products.
\end{prop}

\begin{proof}
    Let $A$ and $B$ be finitely generated subgroup separable groups with the $\Dg$-Hall property. Let $X_A$ and $X_B$ be classifying spaces for $A$ and $B$, respectively, and let $X$ be the space obtained from $X_A$, $X_B$, and an edge $I = [0,1]$ by gluing the point $0 \in I$ to a basepoint in $X_A$ and the point $1 \in I$ to a basepoint in $X_B$. Then $X$ is a classifying space for $A * B$, and has a natural graph of spaces structure, where the underlying graph is an edge.

    Let $H \leqslant A * B$ be a finitely generated subgroup and let $Y \lrar X$ be the covering space corresponding to $H$. Let $Z$ be a finite core for $Y$, i.e.~$Z \longhookrightarrow Y$ induces a $\pi_1$-isomorphism, the underlying graph of $Z$ is finite, and $Z_v = Y_v$ for all vertices $v$ in the underlying graph of $Z$. Denote the fundamental groups of the $A$-vertices (i.e.~those vertex spaces in $Z$ covering $X_A$) by $X_{A_i}$, where $\pi_1(X_{A_i}) = A_i \leqslant A$. Similarly, denote the $B$-vertices by $X_{B_j}$, where $\pi_1(X_{B_j}) = B_j \leqslant B$. For each $i$ (resp.~$j$), let $A_i' \leqslant A$ (resp.~$B_j' \leqslant B$) be a finite-index subgroup containing $A_i$ (resp.~$B_j$) such that $A_i \longhookrightarrow A_i'$ (resp.~$B_j \longhookrightarrow B_j'$) is $\Dg$-injective (recall \cref{def:Dginjective}).

    Let $X_{A_i} \lrar X_{A_i'}$ be the covering map associated to $A_i \leqslant A_i'$ and let $P_i \subseteq X_{A_i}$ be the set of points that are the endpoints of edges $Z$. By subgroup separability of $A$, we may find a finite-index subgroup $A_i''$ such that $A_i \leqslant A_i'' \leqslant A_i'$ and such that the induced covering map $X_{A_i} \lrar X_{A_i''}$ is injective on $P_i$. Note that $A_i$ is still $\Dg$-injective in $A_i''$. A similar discussion applies to the $B$-vertices, where we obtain new groups $B_j''$ and spaces $X_{B_j''}$ satisfying the analogous conditions.

    Let $\overline Z$ be the following graph of spaces: it has the same underlying as $Z$, the vertex spaces $X_{A_i}$ (resp.~$B_j$) are replaced with $X_{A_i''}$ (resp.~$X_{B_j''}$), and there is an edge joining the points $x \in X_{A_i''}$ and $y \in X_{B_j''}$ if and only if they are the images of points $x'$ and $y'$ under the coverings $X_{A_i} \lrar X_{A_i''}$ and $X_{B_j} \lrar X_{B_j''}$, respectively, and $x'$ and $y'$ were joined by an edge in $Z$. From the construction, the covering spaces of the vertices induce a map of graphs of spaces $Z \lrar \overline Z$ (which is an isomorphism on underlying graphs). Then $\pi_1(Z) \longhookrightarrow \pi_1(\overline Z)$ is $\Dg$-injective by \cref{lem:MV}.

    The process of completing $\overline Z$ to a finite-sheeted cover $\widehat Z$ of $X$ is standard. This is detailed, for instance, in \cite[Theorem 3.2]{WiltonPrehall}. For this, one adds various disjoint copies of the vertices $X_A$ and $X_B$ to the precover $\overline Z$ until the resulting space satisfies {\it Stallings' principle} (see \cite[Proposition 3.1]{WiltonPrehall}). Then, certain pairs of the hanging elevations of edge maps can be glued together along additional trivial edge spaces to produce the finite-sheeted cover $\overline Z \lrar X$. As before, the inclusion $\overline Z \longhookrightarrow \widehat Z$  induces a $\Dg$-injection on fundamental groups, which proves the claim. \qedhere
\end{proof}
It is natural to ask whether subgroup separability is needed in \cref{prop:L2freeProd}, but it is unclear to the authors if, for instance, the free product of finitely generated and residually finite $L^2$-Hall groups is  $L^2$-Hall. For non-residually finite groups we notice the following.
\begin{rem} 
    The $L^2$-Hall property is not closed under free products in general. Let $A$ be an infinite, simple, amenable group (finitely generated examples of such groups exist by \cite{JuschenckoMonod_simpleAmenable}). Then $A$ has the $L^2$-Hall property but $A*A$ does not. To see this, let $F \leqslant A * A$ be a free subgroup of rank $d(F) > 2$. Then $\b_1(F)>1=\b_1(A*A)$ and hence $F$ is not $L^2$-independent in $A*A$. Moreover, $A$ is simple and therefore $A*A$ has no nontrivial finite-index subgroups. We conclude that $A*A$ does not have the  $L^2$-Hall property.
\end{rem}

 We conclude with some non-examples.

\begin{ex} 
    Fundamental groups of hyperbolic $3$-manifolds and (nonabelian free)-by-cyclic groups are examples of groups $G$ with $\bG_1(G) = 0$ that contain nonabelian free subgroups. Consequently, they are not $\Dg$-Hall. For a similar reason, non-solvable generalised Baumslag--Solitar groups are not $\Dg$-Hall.
\end{ex}

\subsection{Passing to finite-index overgroups}

In this subsection we prove \cref{overgroups}. This will be crucial when establishing the $L^2$-Hall property for graphs of free groups with cyclic edge groups. \cref{thm:D} from the introduction will follow from \Cref{cor:overgroups} and \Cref{lem:topvanish}.  

\begin{thm} \label{overgroups} 
    Let $G$ be a finitely generated and suppose that $G_1 \leq G$ is a finite-index subgroup and that $H\leq G$ is a finitely generated subgroup such that   $\bH_2(H)=0$. Then the following hold.
    \begin{enumerate}[label=(\arabic*)]
        \item\label{item:overgroupJZ} If $H$ is $\Dg$-independent in $G$, then $H\cap G_1$ is $\Dg$-independent in $G_1.$
        \item\label{item:overgroup} If there exists a finite-index subgroup $H_0\leq H$ such that $H_0$ is $\Dg$-independent in $G_1$, then there exists a finite-index subgroup $G_0\leq G$ containing $H$ as a $\Dg$-independent subgroup.
    \end{enumerate}
\end{thm}

The first statement \ref{item:overgroupJZ} is essentially \cite[Proposition 5.2]{Jai22}, whose argument is followed to additionally prove statement \ref{item:overgroup}. We first prove the following simple lemma. 

\begin{lem} \label{sandwich}
    Let $G$ be a finitely generated group and suppose that $H\leq T\leq G$ are subgroups such that $|T: H|<\infty$. If $H$ is $\Dg$-independent in $G$, then $T$ is $\Dg$-independent in $G.$
\end{lem}
\begin{proof}
    Consider the short exact sequence of $K[G]$-modules 
    \[
        0\lrar I_{T}^G/I_H^G\lrar I_G/I_H^G\lrar I_G/I_T^G\lrar 0.
    \]
    The induced long exact sequence in $\Tor_\bullet^{K[G]}(\Dg, -)$ contains the following sequence of $\Dg$-modules:
    \[
        \Tor_1^{K[G]}(\Dg, I_G/I_H^G)\lrar \Tor_1^{K[G]}(\Dg, I_G/I_T^G) \lrar \Tor_0^{K[G]}(\Dg, I_{T}^G/I_H^G).
    \]
    By \cref{prop:L2indep2} and the assumption that $H$ is $\Dg$-independent in $G$, it follows that the left-most term $\Tor_1^{K[G]}(\Dg, I_G/I_H^G)$ is zero. Moreover, since $H$ is finite index in $T$, it is not hard to see that $I_{T}^G/I_H^G$ is a finite-dimensional $K$-vector space, so $\Tor_0^{K[G]}(\Dg, I_{T}^G/I_H^G)=0$. It follows directly from the short exact sequence above that $\Tor_1^{K[G]}(\Dg, I_G/I_T^G)=0$. This implies, again by \cref{prop:L2indep2}, that $T$ is $\Dg$-independent in $G$. \qedhere
\end{proof}
 We are now ready to explain the proof of \cref{overgroups}.
\begin{proof}[Proof of \cref{overgroups}] 

Let $H_1=H\cap G_1$. We begin by proving statement \ref{item:overgroupJZ}. By \cref{prop:L2indep2}, it is enough to show that 
\[
    \Tor_1^{K[G_1]}(\D_{K[G_1]}, I_{G_1}/I_{H_1}^{G_1})=0.
\]

\begin{claim} \label{inter} 
    As subsets of $K[G]$, we have the equality $I_{H_1}^{G_1}=I_{G_1}\cap I_{H}^{G}.$
\end{claim}
\begin{proof}
    Consider the following commutative diagram of natural maps 
    \[\begin{tikzcd}
        K[G_1] \ar[d, "p_{H_1}^{G_1}"] \ar[r, hookrightarrow, "\iota_1"] & K[G] \ar[d, "p_H^G"]\\
        K[G_1/H_1] \ar[r, hookrightarrow, "\iota_2"] & K[G/H].
    \end{tikzcd}\]
    The horizontal arrows $\iota_1$ and $\iota_2$ are injective. It is clear that $I_{H_1}^{G_1}\subseteq I_{G_1}\cap I_{H}^{G}.$ For the reverse inclusion, we will use the above diagram. If $x\in  I_{G_1}\cap I_{H}^{G} $, then $x\in K[G_1]$ and $x$ belongs to the kernel of $p_H^G\circ \iota_1$. By the commutativity of the diagram and the injectivity of $\iota_2$, the element $x$ must belong to the kernel of $p_{H_1}^{G_1}$, which equals $I_{H_1}^{G_1}$ by \cref{augiso}. \renewcommand\qedsymbol{$\diamond$}
\end{proof}

\cref{inter} implies that the natural map of $K[G_1]$-modules $I_{G_1}/I_{H_1}^{G_1}\lrar I_{G}/I_{H}^{G}$ is injective. Furthermore, since $I_{G_1}/I_{H_1}^{G_1}$ (resp.~$I_{G}/I_{H}^{G}$) is the kernel of the augmentation $K[G_1/H_1]\lrar K$ (resp. $K[G/H]\lrar K$) by \cref{augiso}, there is an exact sequence of $K[G_1]$-modules of the form  
\[
    0\lrar I_{G_1}/I_{H_1}^{G_1} \lrar I_{G}/I_{H}^{G} \lrar K[G/H]/K[G_1/H_1] \lrar 0.
\]
Let $T\subseteq G$ be a set of representatives for the double $(G_1, H )$-cosets in $G$ such that $1\in T$. Denote by $M_t$ the $K[G_1]$-module $K[G_1/(H^t\cap G_1)]$. Then $K[G/H]\cong \oplus_{t\in T} M_t$ as $K[G_1]$-modules. Let $\D=\D_{K[G_1]}$. Notice that $\Tor_2^{K[G_1]}(\D, M_t)=0$ for all $t \in T$. The reason is that, by \cref{prop:HFprops}\ref{item:Shapi0}, its $\D$-dimension equals
\begin{equation} \label{eq:overgroups}
    b_2^{K[H^t\cap G_1]}(H^t\cap G_1)=b_2^{K[H]}(H) \cdot |H^t: H^t\cap G_1| = 0.
\end{equation}
Note that, to obtain (\ref{eq:overgroups}), we have also used the fact that $H^t\cap G_1$ is finite index in $H^t\cong H$, as well as the multiplicativity of $\Dg$-Betti numbers (\cref{prop:HFprops}\ref{item:fibetti}). By the additivity of the $\D$-dimension function, it follows from \cref{prop:HFprops}\ref{item:Shapi0} and \cref{eq:overgroups} that the $K[G_1]$-module $N=K[G/H]/K[G_1/H_1]\cong \oplus_{t\in T \smallsetminus \{1\}} M_t$ has $\Tor_2^{K[G_1]}(\D, N)=0$. 

The long exact sequence in $\Tor$ gives us an exact sequence of $\D$-modules of the form 
\begin{equation} \label{Deq}
    \begin{tikzcd}[sep = small]
        & \cdots  \arrow[r]\arrow[d, phantom, ""{coordinate, name=Z}]& \Tor_2^{K[G_1]}(\D,  N) \arrow[dll,  rounded corners, to path={--([xshift=2ex]\tikztostart.east)|- (Z) [near end]\tikztonodes-| ([xshift=-2ex]\tikztotarget.west)-- (\tikztotarget)}] \\
        \Tor_1^{K[G_1]}(\D,  I_{G_1}/I_{H_1}^{G_1}) \arrow[r]& \Tor_1^{K[G_1]}(\D, I_{G}/I_{H}^{G})  \arrow[r] &  \Tor_1^{K[G_1]}(\D,  N).
    \end{tikzcd}
\end{equation}
We have already proved that $\Tor_2^{K[G_1]}(\D,  N)=0$. So statement \ref{item:overgroupJZ} will follow from diagram \eqref{Deq} if we prove that $\Tor_1^{K[G_1]}(\D, I_{G}/I_{H}^{G})=0$. We know from \cref{prop:HFprops}\ref{item:Shapi1} that
\begin{equation*}
    \Tor_1^{K[G_1]}(\D, I_{G}/I_{H}^{G})   \cong \Tor_1^{K[G]}(\Dg, I_{G}/I_{H}^{G}).
\end{equation*}
Furthermore, the right-hand side vanishes by  \cref{prop:L2indep2} and the assumption that $H$ is $\Dg$-independent in $G$. This completes the proof of the first statement \ref{item:overgroupJZ}.

\smallskip

We now prove \ref{item:overgroup}. The subgroups $H_0\leq H\cap G_1\leq G_1$ have the property that $| H\cap G_1: H_0|<\infty$ and that $H_0$ is $\Dg$-independent in $G_1$. By \cref{sandwich}, $H\cap G_1$ is $\Dg$-independent in $G_1$. Let $G_2 \trianglelefteqslant G_1$ be a normal subgroup of finite index; since $\bH_2(H)=0$, it follows from part \ref{item:overgroupJZ} that $H\cap G_2$ is $\Dg$-independent in $G_2$. Thus, by replacing $G_1$ by $G_2$, we may assume that $G_1$ is normal in $G$ and that $H\cap G_1$ is $\Dg$-independent in $G_1$. 

We claim that $H$ is $\Dg$-independent in $G_0=\lan G_1, H\ran=G_1\cdot H$. For this, we first observe that $T=\{1\}$ is a set of representatives of the double $(G_1, H)$-cosets in $G_0$. So the argument given in \ref{item:overgroupJZ} shows that the canonical map 
\begin{equation}\label{eq:accident}
    I_{G_1}/I_{H\cap G_1}^{G_1}\xrightarrow[]{\cong} I_{G_0}/I_{H}^{G_0}
\end{equation}
is an isomorphism of $K[G_1]$-modules. Using that $H\cap G_1$ is $\Dg$-independent in $G_1$, we can argue as before to deduce from (\ref{eq:accident}), \cref{prop:HFprops}\ref{item:Shapi1} and \cref{prop:L2indep2} that $$  \Tor_1^{K[G_0]}(\mathcal{D}_{K[G_0]}, I_{G_0}/I_{H}^{G_0})\cong\Tor_1^{K[G_1]}(\mathcal{D}_{K[G_1]}, I_{G_1}/I_{H\cap G_1}^{G_1})=0.$$ Thus, again by \cref{prop:L2indep2}, $H$ is $\Dg$-independent in $G_0$ 
\end{proof}

The following is a direct consequence of \cref{overgroups}. 

\begin{cor}\label{cor:overgroups}
    Let $G$ be a finitely generated group and suppose that all finitely generated subgroups $H\leq G$ have the property that $\bH_2(H)=0$ and that there exist finite-index subgroups $H_1\leq H$ and $G_1\leq G$ such that $H_1$ is $\Dg$-independent in $G_1$. Then $G$ is $\Dg$-Hall.
\end{cor}

\Cref{cor:overgroups} offers a more flexible reformulation of the $\Dg$-hall property which will be used in \cref{thm:GOFGWCEG} to establish the $\Dg$-hall property for various graphs of free groups and cyclic edge groups. Moreover, the local condition on the vanishing of $\bH_2(H)=0$ for all $H \leq G$ can be condensed for certain groups of cohomological dimension $2$ using the following lemma.

\begin{lem}\label{lem:topvanish}
    Let $G$ be a group of cohomological dimension $\cd_K(G) = n$ with $b_n^{K[G]}(G) = 0$. Then $b_n^{K[H]}(H) = 0$ for every subgroup $H \leqslant G$.
\end{lem}

\begin{proof}
    Note that the natural map
    \begin{equation} \label{eq:topvanish}
        \Dh \otimes_{K[H]} F \lrar \Dg \otimes_{K[G]} F
    \end{equation}
    is injective, where $F$ is a free left $K[G]$-module. To see this, it is enough to prove the claim when $F =K[G]$. Let $T$ be a right transversal for $H$ in $G$. The Hughes-freeness of $\Dg$ implies that the map $\oplus_{t \in T} \Dh \cdot t \lrar \Dg$ induced by the inclusions $\Dh \cdot t \longhookrightarrow \Dg$ is injective \cite[Corollary 8.3]{Grater20}. The map of (\ref{eq:topvanish}) when $F=K[G]$ equals the composition
    \[
        \Dh \otimes_{K[H]} K[G] \xrightarrow{\cong} \Dh \otimes_{K[H]} \left( \bigoplus_{t \in T} K[H] \cdot t \right)  \xrightarrow{\cong}\bigoplus_{t \in T} \Dh \cdot t \longhookrightarrow \Dg
    \]
    and is therefore injective, as desired.

    The claim now follows easily. Let $0 \lrar F_n \lrar \cdots \lrar F_0 \lrar K \lrar 0$ be a free resolution of the trivial $K[G]$-module $K$. This resolution exists because $G$ has a classifying space of dimension at most $n$ (we do not claim the modules $F_i$ to be finitely generated). If $\bH_n(H) \neq 0$, then there is a non-trivial element $z$ in the kernel of $\Dh \otimes_{K[H]} F_n \lrar \Dh \otimes_{K[H]} F_{n-1}$. Then $z$ is also a nonzero element of the kernel of $\Dg \otimes_{K[G]} F_n \lrar \Dg \otimes_{K[G]} F_{n-1}$ and therefore $\bG_n(G) \neq 0$. \qedhere
\end{proof}

While we only have a conjectural characterisation of which general graphs of free groups with cyclic edge are $L^2$-Hall, the case of an amalgam is entirely understood.

\begin{cor}
    If $G$ is an amalgam of free groups over a cyclic subgroup, then $G$ has the $L^2$-Hall property if and only if it does not contain a subgroup isomorphic to $F_2 \times \Z$.
\end{cor}

\begin{proof}
    First note that $F_2 \times \Z$ is not $L^2$-Hall and so it cannot be a subgroup of an $L^2$-Hall group by \cref{lem:l2subgp}. Conversely, assume that $G$ does not contain a copy of $F_2 \times \Z$. Then \cite[Theorem 1.2]{Wis18} implies that $G$ has a finite-index subgroup that is a limit group.  Limit groups are $L^2$-Hall by \cite{BrownKar2023quantifying} and  have vanishing second $L^2$-Betti number by \cite{BridKoch_volumeGradient}. Thus, $G$ is $L^2$-Hall by \cref{cor:overgroups}.  \qedhere
\end{proof}

\section{Graphs of free groups with cyclic edge groups} \label{sec:GofFG}

\begin{rem}
    For this section and the next, we will focus on the $L^2$-Hall property. Indeed, for graphs of free groups with cyclic edge groups and for limit groups, the $L^2$-Betti numbers and the $\Dg$-Betti numbers coincide. For this reason, and for simplicity, this and the following sections are written in terms of $L^2$-homology. 
\end{rem}

Throughout this section, $G$ will denote the fundamental group of a finite graph of free groups with cyclic edge groups $(G_v, G_e; \Gamma)$ and $X $ will denote the geometric realisation of a corresponding graph of spaces $\mathcal X = (X_v, X_e; \Gamma)$ with $S^1$  edge spaces such that $G = \pi_1(X)$. The attaching maps $X_e \lrar X_v$ are always assumed to be immersions. We will prove the $L^2$-Hall property for some of these groups  in \cref{thm:GOFGWCEG}. Our strategy is as follows.
\begin{itemize}
    \item We allow ourselves to work with clean graphs of groups by \cref{thm:LERF} below and \cref{overgroups}\ref{item:overgroup}.
    \item Given a finitely generated subgroup $H$ of $G$, in order to craft $G_1 \leqslant G$ of finite index and an $L^2$-injective map $H\lrar G_1$, we will use the cyclic splitting of $G$ and  the $L^2$-injectivity criteria for graphs of groups developed in the previous section (such as \cref{lem:MV}). The construction of $G_1$ uses Wise's argument on the subgroup separability of some graphs of free groups with cyclic edge groups \cite{Wise_subgroupSep}.
    \item However, this geometric construction will not be directly applicable to $H$ and $G$, but only to further finite-index subgroup $H_0 \leqslant H$ and $G_0 \leqslant G$, so we will also require \cref{overgroups}\ref{item:overgroup} to reach the same conclusion about $H$ and $G$.
\end{itemize}

We now proceed with the construction. The following definition was introduced by Wise \cite[Definition 4.16]{Wise_subgroupSep}.

\begin{defn}[The weighted graph $\Phi_X$ associated to $X$] \label{def:balanced} 
    A \textit{weighted graph} is a directed graph whose edges have two integer labels (one on each endpoint). A weighted graph $\Gamma$ is \textit{balanced} if whenever $\sigma \colon S^1 \lrar \Phi_X$ is an oriented combinatorial loop (which means that $S^1$ is given a graph structure by subdivision which makes $\sigma$ into a map of graphs), the product of the outgoing weights divided by the product of the incoming weights on $S^1$ equals $\pm 1$ (where the weights on $S^1$ are induced by $\sigma$). Moreover, $\Gamma$ is \textit{solvable} if it can be oriented so that every vertex has at most one outgoing edge and the weight of every incoming edge is $\pm 1$.
    
    We are going to associate to $X = (X_v, X_e; \Gamma)$ a weighted graph $\Phi_X$ that is defined as follows. Fix an orientation for every simple closed combinatorial loop of all vertex spaces $X_v$ and fix an orientation of $S^1$ (say counter-clockwise, viewed as a subset of $\C$). The set of edges of $\Phi_X$ equals the set of edges of $\Gamma$. We identify the endpoint $v$ of $e$ to the endpoint $v'$ of $e'$ if and only if the images of the attaching maps $X_e \longrightarrow X_v$ and $X_{e'} \longrightarrow X_{v'}$ are equal. Let $n$ be the maximal integer such that $X_e \longrightarrow X_v$ represents an $n$th power of an element in $\pi_1(X_v)$. Then, if the attaching map respects orientations we, put a weight of $|n|$ on the end $e$ of $v$; otherwise we put a weight of $-|n|$.
\end{defn}

\begin{defn}\label{def:assocGBS}
    A connected weighted graph $\Gamma$ determines a graph of spaces $X_\Gamma$. as follows. For each vertex (resp.~edge) of $\Gamma$ there is a vertex (resp.~edge) space homeomorphic to $S^1$, all oriented counter-clockwise. The edge spaces are attached to the vertex spaces by degree $n$ covers, where $n$ is the weight on the corresponding end of the edge (we take $n < 0$ to mean that the covering map is of degree $|n|$ in the usual sense and it reverses orientations). We call $\pi_1(X_\Gamma)$ the \textit{generalised Baumslag--Solitar group associated to $\Gamma$} and $X_\Gamma$ the \textit{generalised Baumslag--Solitar complex associated to $\Gamma$}.
\end{defn}

\begin{lem}\label{lem:embedGBS}
    Let $X$ be a graph of free groups with cyclic edge groups and let $\Gamma$ be a component of the weighted graph $\Gamma$. Let $G$ be the generalised Baumslag--Solitar group associated to $\Gamma$. Then the natural map $G \lrar \pi_1(X)$ is $\pi_1$-injective.
\end{lem}

\begin{proof}
    Fix normal forms for elements in $G$ and $\pi_1(X)$. It is then not hard to see that elements of $G$ in normal form are sent to elements of $\pi_1(X)$ in normal form. \qedhere
\end{proof}

\begin{ex}
    The Baumslag--Solitar group $\BS(m,n)$ is the fundamental group of a graph of spaces of the form $(S^1, S^1; \Gamma)$ where $\Gamma$ is a single loop and the two attaching maps are degree $m$ and $n$ covering maps $S^1 \longrightarrow S^1$. In this case $\Phi_X$ is a loop with one vertex and one edge, where the ends of the edge are labeled by $m$ and $n$. Note that $\Phi_X$ is balanced if and only if $m = \pm n$ and it is solvable if and only if $m = \pm 1$ or $n = \pm 1$.
\end{ex}

Many properties of graphs of free groups with cyclic edge groups can be characterised by the properties of $\Phi_X$. The following definition and result are due to Wise.

\begin{defn}[Clean graph of spaces] \label{def:clean}
    A graph of spaces is \textit{clean} if every edge map is a topological embedding.
\end{defn}

\begin{thm}[{\cite[Theorems 4.18 and 5.1]{Wise_subgroupSep}}]\label{thm:LERF}
    Let $G = \pi_1(X)$ be a finitely generated graph of free groups with cyclic edge groups. The following are equivalent:
    \begin{enumerate}
        \item $G$ is subgroup separable;
        \item $\Phi_X$ is balanced;
        \item the generalised Baumslag--Solitar groups associated to the components of $\Phi_X$ are all subgroup separable;
        \item $X$ has a finite clean cover.
    \end{enumerate}
\end{thm}

We highlight the following recent result of Abgrall and Munro that confirms a conjecture of Wise \cite[Conjecture 6.2]{Wise_subgroupSep} and gives an easily computable criterion for when a graph of free groups with cyclic edge groups is residually finite.

\begin{thm}[{\cite{AbgrallMunro_RFgraphsOfFree}}]\label{thm:RF}
    Let $G = \pi_1(X)$ be a finitely generated graph of free groups with cyclic edge groups. The following are equivalent:
    \begin{enumerate}
        \item $G$ is residually finite;
        \item every component of $\Phi_X$ is balanced or solvable;
        \item the generalised Baumslag--Solitar groups associated to the components of $\Phi_X$ are all residually finite.
    \end{enumerate}
\end{thm}

The main goal of this section is to establish the following theorem.

\begin{thm}\label{thm:GOFGWCEG}
    Let $G$ split as a finitely generated graph of free groups with cyclic edge groups and let $G = \pi_1(X)$, where $X$ is as above. If $\Phi_X$ is balanced and solvable, then $G$ has the $L^2$-Hall property. Equivalently, if $G$ is hyperbolic relative to virtually abelian subgroups, then $G$ has the $L^2$-Hall property.
\end{thm}

\begin{rem}
    The condition that every component of $\Phi_X$ be solvable is necessary, since otherwise $G$ would contain a non-solvable Baumslag--Solitar subgroup. Such groups do not have the $L^2$-Hall property since they contain nonabelian free subgroups but have vanishing first $L^2$-Betti number. On the other hand, there are graphs of free groups with cyclic edge groups where $\Phi_X$ is unbalanced yet $G$ is still $L^2$-Hall (e.g.~$\BS(1,n)$ for $n \neq \pm 1$).
\end{rem}

 Motivated by this remark, we make the following conjecture, which is formally similar to \cref{thm:LERF,thm:RF}.

 \begin{conj}\label{conj:F2Z}
    Let $G = \pi_1(X)$ be a finitely generated graph of free groups with cyclic edge groups. The following are equivalent:
    \begin{enumerate}
        \item $G$ has the $L^2$-Hall property;
        \item every component of $\Phi_X$ is solvable;
        \item the generalised Baumslag--Solitar groups associated to the components of $\Phi_X$ all have the $L^2$-Hall property.
    \end{enumerate}
\end{conj}

\subsection{The proof of \texorpdfstring{\cref{thm:A}}{Theorem A}} We prove \cref{thm:GOFGWCEG} above (which is \cref{thm:A} from the introduction). We make a few simplifying reductions that we hope will make the visualisation of the objects easier as well as put us in the context of the proof of \cite[Theorem 5.2]{Wise_subgroupSep}.

\begin{claim} \label{claim:Z} 
    It is enough to consider the case where all the edge groups are infinite cyclic.
\end{claim} 
\begin{proof}
    A balanced graph of free groups with cyclic edge groups is the free product of balanced graphs of free groups all whose edge groups are infinite cyclic (all of which are subgroup separable by \cref{thm:LERF}). By \cref{prop:L2freeProd}, it is enough to prove that each free factor is $L^2$-Hall, hence the claim. \renewcommand\qedsymbol{$\diamond$} \qedhere
\end{proof}

\begin{claim}\label{claim:clean}
    It is enough to prove \cref{thm:GOFGWCEG} in the case where $X$ is clean.
\end{claim}
\begin{proof} 
    By \cite[Lemma 4.4 and Theorem 4.18]{Wise_subgroupSep}, $X$ has a clean finite-sheeted covering $X_c \lrar X$. By \cref{cor:overgroups}, the $L^2$-Hall property passes to finite-index overgroups, so it is enough to prove \cref{thm:GOFGWCEG} for $\pi_1(X_c)$. \renewcommand\qedsymbol{$\diamond$} \qedhere
\end{proof}

\begin{defn}
    Let $X$ be a clean graph of free groups with cyclic edge groups. An \textit{immersed Klein bottle} in $X$ is a subcomplex $K \subseteq X$ that corresponds to a loop in $\Phi_X$ whose associated generalised Baumslag--Solitar group is a Klein bottle group. Similarly, an immersed torus is a subcomplex corresponding to a loop in $\Phi_X$ whose associated generalised Baumslag--Solitar group is $\Z^2$.
\end{defn}

Note that an immersed Klein bottle $K$ is indeed the image of a Klein bottle surface $S$ under a cellular immersion, where $S$ is the generalised Baumslag--Solitar complex associated to the loop corresponding to $K$. Similarly, if $T$ is an immersed torus, then there is a topological torus $S$ and a cellular immersion $S \lrar T$.

\begin{claim}
    It is enough to prove \cref{thm:GOFGWCEG} in the case where $X$ is clean and does not contain any immersed Klein bottles.
\end{claim}

\begin{proof}
    By \cref{claim:clean}, we may assume that $X$ is clean. Let $K_1, \dots, K_n$ denote the immersed Klein bottles in $X$. By subgroup separability of $G = \pi_1(X)$, there is a finite-sheeted regular cover $p_1 \colon X_1 \longrightarrow X$ where the components of $p\inv(K_1)$ are all immersed tori. Assume now that we have constructed some finite-sheeted regular cover $p_i \colon X_i \longrightarrow X$ so that for each $j \leqslant i$ every component of $p_i\inv(K_j)$ is an immersed torus. Let $K$ be an immersed Klein bottle component of $p_i\inv(K_{i+1})$. Again we may pass to a further finite-sheeted cover $q_{i+1} \colon X_{i+1} \longrightarrow X_i$ such that the composition
    \[
        p_{i+1} \colon X_{i+1} \longrightarrow X_i \longrightarrow X
    \]
    is regular and every component of $q_{i+1}\inv(K)$ is an immersed torus. But then every component of $p_{i+1}\inv(K_{i+1})$ is an immersed torus by regularity of the cover. By \cref{cor:overgroups}, it is enough to show that $\pi_1(X_n)$ has the $L^2$-Hall property. \renewcommand\qedsymbol{$\diamond$} \qedhere
\end{proof}

\begin{proof}[Proof (of \cref{thm:GOFGWCEG})]
    By the claims above, we assume without loss of generality that $G$ is the fundamental group of a clean graph of spaces $X = (X_v, X_e; \Gamma)$, where the vertex spaces $X_v$ are graphs and the edge spaces $X_e$ are circles. Moreover, we assume that $X$ does not contain any immersed Klein bottles.
    
    Let $H$ be a finitely generated subgroup of $G$. Following the proof of \cite[Theorem 5.2]{Wise_subgroupSep}, we will show that there is a finite-index subgroup $H_1 \leqslant H$ that is $L^2$-independent in a finite-index subgroup $G_1 \leqslant G$. By \cref{overgroups}, this is enough to prove the theorem. We break up our proof into steps in a similar way to the proof of \cite[Theorem 5.2]{Wise_subgroupSep}.
    
    Let $Y \longrightarrow X$ be the covering space corresponding to $H$. Note that $Y$ has a natural decomposition into a clean graph of spaces $(Y_v, Y_e; \Gamma_Y)$, where each of the vertex spaces are graphs and each of the edge spaces are homeomorphic to either $S^1$ or $\R$.

    \setcounter{step}{0}
    
    \begin{step}[The subcomplex]
        Denote the underlying graph of $Y$ by $\Gamma(Y)$. Since $H$ is finitely generated, there is a finite connected subgraph $\Upsilon$ of $\Gamma(Y)$ such that the inclusion $Y_\Upsilon \longhookrightarrow Y$ of the restricted graph of spaces $Y_\Upsilon$ is a $\pi_1$-isomorphism.  
    \end{step}

    \begin{step}[Pruning]
        For each vertex space $Y_v$ of $Y_\Upsilon$, let $Z_v$ be the smallest connected subgraph containing the images of all the edge spaces of $Y_\Upsilon$ and such that $Z_v \longhookrightarrow Y_v$ induces a $\pi_1$-isomorphism. Let $Z \subseteq Y_\Upsilon$ be the union of the spaces $Z_v$ and the edge spaces $Y_e$ of $Y_\upsilon$. Note that $Z$ is connected and has a natural graph of spaces structure $(Z_v, Z_e = Y_e; \Upsilon)$ such that the inclusion $Z \longhookrightarrow Y_\Upsilon$ still induces a $\pi_1$-isomorphism. The resulting vertex spaces of $Z_v$ have a compact core with pairs of infinite rays attached to them coming from the attaching maps of non-compact edge spaces in $Y$ (see \cref{fig:pruned}).
        
        \begin{figure}[ht]
            \centering
            \includegraphics[scale=0.5]{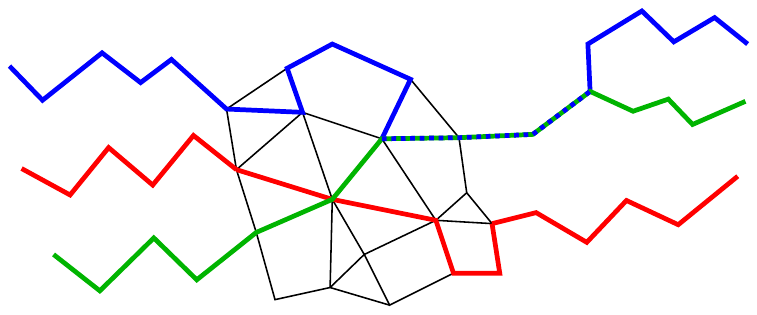}
            \caption{A vertex space of $Z$. The thickened lines represent attaching maps of non-compact edge spaces, each of which being homeomorphic to $\R$.}
            \label{fig:pruned}
        \end{figure}
    \end{step}

    \begin{step}[$L^2$-independence of periphery closing]
        This is the main step of the proof. Let $e$ be an edge of $\Upsilon$ and let $I = [0,1]$ be the closed unit interval. If $Z_e \cong S^1$, then we  call $Z_e \times I$ a \textit{cylinder}; if $Z_e \cong \R$, then $Z_e \times I$ is a \textit{strip}. If two strips in $Z$ have a non-compact intersection, then the periodicity of the attaching maps implies that their intersection must in fact be homeomorphic to $\R$. 
    
        Note that $\Z$ acts on each of the strips by covering translations (where the covering refers to a strip in $Z$ covering a cylinder in $X$). As in \cite[Theorem 5.2, Step 3]{Wise_subgroupSep}, choose $n$ large enough so that, for any edge $e$ corresponding to a strip, all the vertices of $Z_e \times I$ with valence at least $3$ (the valence is counted in the vertex graphs adjacent to the strip) are a distance less than $n$ apart. Then quotient the strips of $Z$ by the action of $n\Z$ to form a new complex $A = Z / {\sim}$. Now $A$ is a clean compact graph of graphs with $S^1$ edge spaces. A typical vertex space is shown in \cref{fig:closedup}.

        \begin{figure}[ht]
            \centering
            \includegraphics[scale = 0.5]{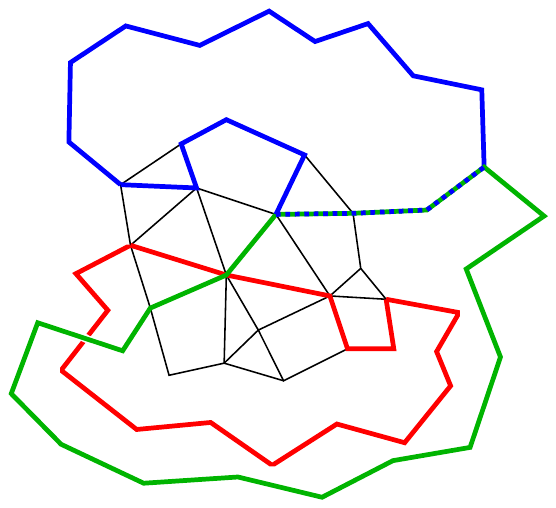}
            \caption{A vertex space of $A$. The vertex space is obtained from that in \cref{fig:pruned} by quotienting the thickened lines by the action of $n\Z$.}
            \label{fig:closedup}
        \end{figure}

        We need to introduce a definition based on one given by Hsu and Wise \cite[Definition 9.1]{HsuWise_graphOfFree}. Declare two strips to be \textit{equivalent} if their intersection is a copy of $\R$; this rule generates an equivalence relation on the set of strips in $Z$. The \textit{periphery} containing a strip $S$ is the union of the strips in the equivalence class of $S$.
        
        Fix a periphery $P$. By the assumption that $X$ contains no immersed Klein bottles, there is a compact subset $K \subseteq Z$ such that $P \smallsetminus K$ is homeomorphic to two disjoint copies of $\R_{>0} \times \Omega$, where $\Omega$ is some finite graph. The effect of quotienting by the action of $n\Z$ can then be rephrased as follows: choose $K$ compact and sufficiently large so that $K$ is a compact core for $Z$ and all the vertices in $P \smallsetminus K$ are of degree $2$. We also require, for every strip $S \subseteq P$, that $K \cap S$ be a fundamental domain for the action of $n\Z$ on $S$. Denote by $R_1$ and $R_2$ the copies of $\R_{>0} \times \Omega$ in $P \smallsetminus K$. We then form quotient of the complex $Z \smallsetminus (R_1 \sqcup R_2)$ by identifying the two copies of $\partial R_1 \cong \partial R_2 \cong \Omega$ (see \cref{fig:periph}). Performing this process for each periphery yields the complex $A$.

        \begin{figure}[ht]
            \centering
            \includegraphics[scale = 0.5]{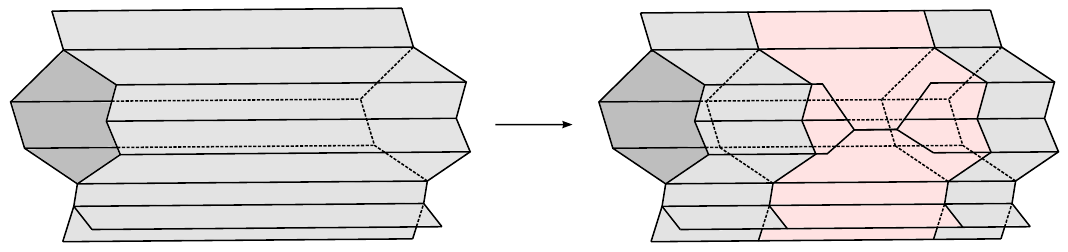}
            \caption{Part of a periphery $P$ is shown on the right. The shaded region represents $K \cap P$, where  $K \subseteq Z$ is as above. The horizontal lines are contained in vertex spaces of $Z$. Outside of $K$, the horizontal lines do not intersect since all the vertices there are of degree $2$. However, they may have a compact intersection inside $K$ as shown in the figure. In this figure, the graph $\Omega \cong \partial R_1 \cong \partial R_2$ is a cycle with two finite trees hanging off it. On the left is part of a copy of $\R \times \Omega$. The whole diagram represents an immersion $\R \times \Omega \lrar P$, which is an isomorphism outside of a compact set.}
            \label{fig:periph}
        \end{figure}

        \begin{claim}
            The injections $\partial R_i \longhookrightarrow P$ are $\pi_1$-injective.
        \end{claim}

        \begin{proof}
            First note that there is a cellular immersion $\R \times \Omega \lrar P$, which is an isomorphism outside of a compact set (see \cref{fig:periph}). The immersion fits into the commutative diagram
            \[
                \begin{tikzcd}
                    \R \times \Omega \arrow[r] \arrow[d] & P \arrow[d] \\
                    S_0 \arrow[r]                        & T_0        
                \end{tikzcd}
            \]
            where the vertical maps are covering spaces, $S_0$ is the graph of spaces of a generalised Baumslag--Solitar group, and $T_0$ is its image in $X$. Covering maps are $\pi_1$-injective and so is the map $S_0 \lrar T_0$ by \cref{lem:embedGBS}. Hence, $\R \times \Omega \lrar P$ is $\pi_1$-injective. Since $\partial R_i \longhookrightarrow \R \times \Omega$ is a $\pi_1$-isomorphism for $i = 1,2$, it follows that the maps $\partial R_i \longhookrightarrow P$ are $\pi_1$-injective. \renewcommand\qedsymbol{$\diamond$} \qedhere
        \end{proof}

        \begin{claim}\label{claim:periph}
            The quotient map $q \colon Z \longrightarrow A$ is $\pi_1$-injective and $q_*(\pi_1(Z))$ is $L^2$-independent in $\pi_1(A)$.
        \end{claim}

        \begin{proof}
            Denote the peripheries of $Z$ by $P_1, \dots, P_n$ and for each $i = 1, \dots, n$ let $\Omega_i$ be the finite graph such that there is an immersion $\R \times \Omega_i \lrar P_i$. Since the peripheries are subgraphs of spaces of $Z$, the groups $\pi_1(\Omega_i) \cong \pi_1(\partial R_i)$ embed in $\pi_1(Z)$ by the previous claim. The quotient map induces an inclusion
            \[
                q_* \colon \pi_1(Z) \longrightarrow \pi_1(Z)*_{\pi_1(\Omega_1), \dots, \pi_1(\Omega_n)} \cong \pi_1(A), 
            \]
            where $\pi_1(Z)*_{\pi_1(\Omega_1), \dots, \pi_1(\Omega_n)}$ denotes the multiple HNN extension of $\pi_1(Z)$ over the subgroups $\pi_1(\Omega_i)$. 
            
            The assumption that $\Phi_X$ is balanced and solvable implies that $G$ does not contain any non-abelian generalised Baumslag--Solitar subgroups. Hence, $\pi_1(\Omega_i)$ is either trivial or isomorphic to $\Z$. To see this, note that every periphery $P_i$ covers a generalised Baumslag--Solitar subcomplex $V_i \subseteq X$, and therefore $\pi_1(V_i)$ cannot contain a nonabelian free subgroup. So from the fact that $\pi_1(\Omega_i) \leqslant \pi_1(V_i)$ we deduce that $\pi_1(\Omega_i)$ is either trivial or $\Z$. The vertex group of a multiple HNN extension along trivial or infinite-cyclic subgroups is $L^2$-independent by \cref{lem:MV}.\renewcommand\qedsymbol{$\diamond$} \qedhere
        \end{proof}
    \end{step}

    \begin{step}[$L^2$-independence of vertex completion]
        As remarked in the previous step, $A$ is a clean compact graph of graphs with $S^1$ edge spaces. Moreover, since the $\Z$ action on the strips was by covering translations, it follows that there is a natural quotient map $A \longrightarrow X$ whose restriction to every vertex space of $A$ is an immersion. By adding edges to the vertex spaces of $A$, we can complete them to coverings of the corresponding vertex spaces of $X$. Denote the complex obtained from $A$ in this way by $B$ and note that $\pi_1(B) \cong \pi_1(A) * F$, where $F$ is a free group. Then $\pi_1(A)$ (and therefore $H$) is $L^2$-independent in $\pi_1(B)$.
    \end{step}

    \begin{step}[Passing to finite-index and completing to a cover]
        In \cite[Theorem 5.2, Steps 5 and 6]{Wise_subgroupSep}, Wise shows how to pass to a finite-sheeted cover $B_1 \longrightarrow B$ which can be completed to a finite-sheeted cover $X_1 \longrightarrow X$ by attaching cylinders to $B_1$.

        \begin{claim}
            $H \cap \pi_1(B_1)$ is $L^2$-independent in $\pi_1(B_1)$.
        \end{claim}
        \begin{proof}
            Since $\pi_1(A)$ is a free factor of $\pi_1(B) \cong \pi_1(A) * F$, we have that $\pi_1(A) \cap \pi_1(B_1)$ is a free factor of $\pi_1(B_1)$ by Kurosh's theorem \cite[Th\'eor\`eme 14, Chaptire I, \S5]{Serre_arbres}. So  $\pi_1(A) \cap \pi_1(B_1)$ is $L^2$-independent in $\pi_1(B_1)$ and hence it suffices to prove that $H \cap \pi_1(B_1)$ is $L^2$-independent in $\pi_1(A) \cap \pi_1(B)$.
    
            The proof of \cref{claim:periph} shows that $\pi_1(A)$ has a graph of groups decomposition with underlying graph a rose, where the unique vertex group is $H$ and the edge groups are either trivial or $\Z$. Then $\pi_1(A) \cap \pi_1(B_1)$ also has a graph of groups decomposition with edge groups either trivial or $\Z$, and $H \cap \pi_1(B_1)$ is a vertex group in this decomposition. By \cref{lem:l2subgp}, $H \cap \pi_1(B_1)$ is $L^2$-independent in $\pi_1(A) \cap \pi_1(B_1)$. \renewcommand\qedsymbol{$\diamond$} \qedhere
        \end{proof}
    
        As mentioned above, Wise shows that we can attach cylinders to $B_1$ to obtain a finite-sheeted covering $X_1 \longrightarrow X$. Therefore $\pi_1(X_1)$ is a multiple HNN extension of $\pi_1(B_1)$ over cyclic subgroups, so $\pi_1(B_1)$ (an thus $H \cap \pi_1(B_1)$ as well) is $L^2$-independent in $\pi_1(X_1)$. In summary, $H\cap \pi_1(B_1)$ has finite index in $H$ and $H\cap \pi_1(B_1)$ is $L^2$-independent in $\pi_1(X_1)$, which has finite index in $\pi_1(X)$. We conclude that $G$ has the $L^2$-Hall property by \cref{cor:overgroups}. \qedhere
    \end{step}
\end{proof}

\section{\texorpdfstring{$L^2$}{L²}-Hall property for limit groups} \label{sec:ICE}

Wilton \cite{WiltonHall} proved that limit groups have the local retractions property (and hence that they are subgroup separable) using Kharlampovich and Miasnikov's \cite{KharlampovichMyasnikov1998} characterisation of limit groups in terms of {\it ICE groups} (see \cref{def:ICE} below). Limit groups are exactly the finitely generated groups that arise as subgroups of ICE groups. Since the local retractions property passes to subgroups, Wilton only needs to deal with ICE groups in \cite{WiltonHall}. Analogously, we certified that the $L^2$-Hall property passes to subgroups in \cref{lem:l2subgp}, so it also sufficient to deal with ICE groups in order to prove \cref{thm:B}. Our argument is different from that of \cite{BrownKar2023quantifying} and we expect it to be flexible enough to include more general finite abelian hierarchies of relatively hyperbolic groups as in \cref{conj:specialyHypSHNC}.

Just as we followed Wise's argument on subgroup separability of balanced graphs of free groups in the previous section, here we follow the ideas developed by Wilton \cite{WiltonPrehall, WiltonHall} on the subgroup separability of limit groups. 

\begin{defn}\label{def:ICE}
    Let $G$ be a group and let $Z \leqslant G$ be the centraliser of an element. The group $G *_Z (Z \times \Z^n)$ is an \textit{extension of $G$ by a centraliser}. A group is an \textit{ICE group} (standing for ``iterated centraliser extension") if it can be obtained from a finitely generated free group by a finite sequence of extensions by a centraliser.
\end{defn}

If $G$ is an ICE group, then it has a classifying space $X$ that can be described as follows. If $G$ is finitely generated and free, then take $X$ to be a bouquet of circles. Otherwise, write $G = H *_Z (Z \times \Z^n)$ for simpler ICE group $H$. It is not hard to show that we may assume that $Z$ is infinite cyclic (see \cite[Remark 1.14]{WiltonHall}). Then take $X$ to be the graph of $Y$ and $T^{n+1}$ with edge group $S^1$, where $Y$ is the classifying space of $H$ constructed by induction, and $S^1$ maps to a loop representing the centralised element in $H$ and to a coordinate circle in $T^{n+1}$. The spaces obtained in this way will be called \textit{ICE spaces}. We refer the reader to \cite[Section 1.6]{WiltonHall} for a concise survey of this material. We emphasise the following important theorem of Kharlampovich and Miasnikov, which gives a powerful characterisation of limit groups.

\begin{thm}[{\cite{KharlampovichMyasnikov1998}}]\label{thm:limit_are_subgp_of_ICE}
    A finitely generated group $G$ is a limit group if and only if it is a subgroup of an ICE group.
\end{thm}

\begin{defn} \label{defnind}
    A collection of elements $\mathcal L$ in a group $G$ is {\it independent} if $g$ commutes with no conjugate of $h$ for all pairs of distinct elements $g, h\in \mathcal L$. We also say that a collection $\mathcal L$ of loops in a space $X$ is independent if they represent an independent collection of elements of $\pi_1(X)$ in the previous sense. 
\end{defn}

\subsection{From graphs of free groups to limit groups}
Before going into details of the work of Wilton, we first revisit   Wise's argument from \cref{sec:GofFG} to explain what are the main difficulties involved when dealing with ICE groups. We should remark that, in the context of limit groups, the process of getting virtually clean covers is hidden in the inductive argument and will not be mentioned again.  

We denote by $X$  a graph of spaces whose underlying graph has two vertices and one edge. We either have that the vertex spaces are graphs, as in Wise's setting, or an ICE space $Y$ and a torus $T^n$, which is the case of interest of this section. The edge space of $X$ is homeomorphic to a circle and the edge maps are assumed to be injective.   Let $H$ be a finitely generated subgroup of $\pi_1(X)$ and let $X_H\lrar X$ be the covering corresponding to $H$. Scott's criterion \cite[Lemma 1.3]{WiltonHall} topologically reformulates the subgroup separability of $\pi_1(X)$ as the ability to complete the precover $X_H$ of $X$ to a finite-sheeted cover $\hat{X}\lrar X$ so any prescribed finite subcomplex $\Delta$ of $X_H$ projects homeomorphically into $\hat{X}$.
\begin{itemize}
    \item The problem of ``pruning'' in Wise's argument  corresponds to taking a core $X'$ of $X_H$ that contains $\Omega$, the compact cores of the fundamental groups of each vertex space of $X_H$ and all the infinite degree elevations of edge maps (i.e.~the infinite strips). An important property of those elevations was that they escape any compact subset of the free splitting of the vertex space they belong to (because they act freely on the vertex graph).  This property is called ``properness'' (as introduced in \cite[Definition 2.12]{WiltonHall}) and is not satisfied by \textit{elliptic} loops. Hence in this setting the pruning must be performed more carefully. Conveniently, as a consequence of the 2-acylindricity of the Bass--Serre tree of an ICE group, non-elliptic (i.e. \textit{hyperbolic}) loops are proper \cite[Lemma 2.16]{WiltonHall}. 
    \smallskip
    \item The second part of Wise's argument, which consists in closing up the infinite strips, is another step towards obtaining a finite-sheeted precover from the precover $X'$ (since, after this, the preimages of points in the edge spaces are finite).  However, one still has to figure out how to complete the precover so preimages of points in the vertex spaces are finite as well. This corresponds to the problem of extending finite-sheeted covers of the edge spaces to finite-sheeted covers of the vertex spaces themselves. In this setting, this relies on a primitive version of omnipotence of free groups. This way, one constructs the precover $W \lrar X'$.  Slightly more general conditions are offered  in \cite[Section 3.2]{WiltonPrehall} in terms of homological assumptions on the edge spaces. However, for the case when $Y$ is an ICE space and the edge subspace is generic, this problem is resolved in \cite{WiltonHall} by strengthening the inductive hypothesis (incorporating the notion of \textit{tameness}), so to have the required control on the prescribed collection of infinite-degree elevations. 
    \smallskip
    \item Lastly, the finite intermediate precover $W$ is shown to admit a finite-sheeted covering $W_m\lrar W$ that can be completed to a finite-sheeted covering $X_m\lrar X$. This is done similarly in Wilton's argument when $Y$ is an ICE space using the inductive hypothesis, without passing to a deeper $W_m$.
\end{itemize}
The second point above explains Wilton's observation \cite[Section 3]{WiltonHall} that the properties of local retractions of subgroup separability are not strong enough to serve as an induction hypothesis. We notice a similar problem. Following Scott's philosophy, the $L^2$-Hall property concerns the ability to complete precovers of $X$ to finite-sheeted covers preserving the $L^2$-homology in the process. Our following example gives another reason for why this is not strong enough for an inductive argument either. 

\begin{ex} \label{L2tamex}
    Consider $G=\pi_1(\Sigma_2) = \langle a, b, c, d \mid [a, b]=[c, d]\rangle$, which splits as $F(a, b)*_{[a,b]=[c, d]}F(c, d)$. We consider the $L^2$-independent subgroups $H\leq F(a, b)$ and $K\leq F(c, d)$ given by $H=F(a^2, b^2)$ and $K=F(c^2, d^2)$. It is clear that the induced map $H*K\lrar G$ is injective, although it is not $L^2$-injective for the obvious reason that $\b_1(H*K)=3>2=\b_1(G)$.
\end{ex}

 \cref{L2tamex}  illustrates that subgraphs of groups that are $L^2$-injective on vertex groups need not be $L^2$-injective overall, and so one needs some control on the non-trivial $L^2$-classes that have non-trivial support on multiple vertex spaces. Wilton's notion of tameness \cite[Definition 3.1]{WiltonHall} and \cref{L2tamex} motivates the following notion of $L^2$-tameness that allows us to inductively have such control.

\begin{defn}\label{L2tame}
    Consider a complex $X$, a covering $X'\lrar X$ and a finite (possibly empty) collection of independent (\cref{defnind}), essential loops $\mathcal L=\{\delta_i\po C_i\lrar X\}.$ The cover $X'$ is \textit{$L^2$-tame over $\mathcal L$} if the following holds. Let $\Delta\subseteq X'$ be a finite subcomplex and let $\{\delta_j'\po C_j'\lrar X'\}$ be a finite collection of (pairwise non-isomorphic) infinite degree  elevations where each $\delta_j'$ is an elevation of some $\delta_i\in  \mathcal L$. Then for all sufficiently large positive integers $d$, there exists an intermediate finite-sheeted covering $X'\lrar \widehat{X}\lrar X$ that satisfies the following.
    \begin{enumerate}
        \item Every $\delta_j'$ descends to an elevation $\widehat{\delta}_j$ of degree $d$; 
        \item The $\widehat{\delta}_j$ are pairwise non-isomorphic; 
        \item The subcomplex $\Delta$ injects into $\widehat{X}$;  
        \item The induced map 
        \[\pi_1(X')*\left(\coprod_{\mathcal L}\Z\right)\lrar \pi_1(\widehat{X}),\]
        defined by mapping   each $\Z$ labelled by $i$ to the class of the image of $ \widehat{\delta}_i$, is injective and $L^2$-injective. 
    \end{enumerate}
\end{defn} 

\begin{rem}
    The $L^2$-tameness of all coverings $X'\lrar X$ with finitely generated $\pi_1(X')$ and empty $\mathcal L$ implies the $L^2$-Hall property for $\pi_1(X)$.
\end{rem}

As anticipated, the idea is that the strengthened version described in \cref{L2tame} (with additional prescribed data relative to $\mathcal L$) admits a proof by induction and avoids bad embeddings as the one described in \cref{L2tamex}.

\subsection{The proof of \texorpdfstring{\cref{thm:B}}{Theorem B}}

By the previous discussion, the following theorem implies \cref{thm:B} from the introduction, and its proof will occupy the remainder of this section. 

\begin{thm}\label{thm:L2tame}
    Let $X$ be an ICE space, let $H\leq  \pi_1(X)$ be a finitely generated subgroup and let $X_H\lrar X$ be the corresponding covering. Suppose that $\mathcal L$ is a (possibly empty) finite set of independent loops each generating a maximal abelian subgroup of $\pi_1(X)$. Then $X_H$ is $L^2$-tame over $\mathcal L$.
\end{thm}
\begin{proof}
    We proceed by induction on the complexity of the ICE space. The base of the induction is the case when $X$ is a graph, which is essentially the classical M. Hall theorem (see \cite[Corollary 1.8]{WiltonHall} for a precise proof). Now assume $X$ is an ICE space that decomposes as a graph of spaces with two vertices (a lower complexity ICE space $Y$ and a torus $T^n$) and one edge space homeomorphic to $S^1$. This naturally induces a graph of spaces structure for $X_H$, whose underlying graph we denote by $\Gamma(X_H)$. Each vertex space of this splitting is either a covering space of $Y$ or a covering space of the torus $T^n$. 
    
    Denote by $\{\delta_i \colon D_i \lrar X\}$ and $\{\varepsilon_i \colon E_i \lrar X\}$ the hyperbolic and elliptic loops of $\mathcal L$, respectively, relative to the splitting of $X$. 

    \setcounter{step}{0}
    \begin{step}[The precovers $X'$ and $X''$]
        Let $\Delta \subseteq X_H$ be a finite subcomplex and let $\{\delta_j^H\}$ and $\{\varepsilon_j'\}$ denote fixed sets of infinite degree elevations of hyperbolic and elliptic loops, respectively, in $\mathcal L$. We begin by taking a subcomplex $X'\subseteq X_H$ that satisfies the following conditions:
        \begin{enumerate}
            \item The subcomplex $X'$ is a core for $H$, i.e.~ $X'$ is a subgraph of spaces with finite underlying graph such that the induced map $\pi_1(X')\lrar \pi_1(H)$ is an isomorphism.
            \item The subcomplex $\Delta$ is contained in $X'$. 
            \item The image of each $\varepsilon_j'$ is contained in $X'$.
            \item Each infinite-degree elevation $\delta_j^H\po \R\lrar X_H$ restricts to a (possibly non-full) elevation $\delta_j'\po D_j'\lrar X'$, where $D_j'\subseteq \R$ is a finite union of compact intervals. 
        \end{enumerate}
        The subscripts $i$ and $j$ are different, indicating that there may be several elevations $\delta_j^H$ for each $\delta_i$, and likewise for $\varepsilon_i$. We will keep completing the precover $X'$ further to get some more intermediate precovers $X'\subset X''\subset \overline{X}\subset X_H$ of $X$. 
        
        From \cite[Lemma 2.24]{WiltonHall}, we can enlarge $X'$ to $X'' \subseteq X_H$ so that $X''$ still enjoys properties (1)---(4) listed above while additionally satisfying that the corresponding elevations $\delta_j''\po D_j''\lrar X''$ are disparate (in the sense of \cite[Definition 2.2]{WiltonHall}). In particular, the induced map 
        \begin{equation}
            \label{eq0} \pi_1(X'') \lrar \pi_1(X_H)
        \end{equation}
        is an isomorphism. 
    \end{step}

    \begin{step}[The precover $\overline X$]
        We shall not need the definition of disparity but, instead, we will explain how this condition is used to extend the precover $X''$ further. Recall that each $D_j''$ is the union of finitely many compact intervals and that $D_j''$ fits in the following commutative diagram:
        \[
            \begin{tikzcd}
                X''\ar[d, hook]         & D_j'' \arrow[d, hook]  \arrow[l, "\delta_j''"'] &  \\
                X_H \ar[d] & \widetilde{D_i}   \arrow[l, "\delta_j^H"']   \arrow[d] \\
                X & \arrow[l, "\delta_i"'] D_i \nospacepunct{.}
            \end{tikzcd} 
        \]
        
        For all sufficiently large positive integers $d$, there exists $\overline{D}_j\cong S^1$ such that $D_j'' \lrar D_i$ factors through an embedding   $D_j''\longhookrightarrow \overline{D}_j$ and a $d$-sheeted covering map $\overline{D}_j\lrar D_i$.
        By \cite[Lemma 2.23]{WiltonHall}, we can extend $X''$ to a precover $\overline{X}$ such that each $\delta_j''$ extends to a full elevation $\overline{\delta}_j \po \overline{D}_j \lrar \overline{X}$ and the diagram
        \[
            \begin{tikzcd}
                X''  \ar[d]         & D_j'' \arrow[d, hook]  \arrow[l, "\delta_j''"'] &  \\
                \overline{X} \ar[d] & \overline{D}_j \arrow[l, "\overline{\delta}_j"'] \arrow[r, "\cong"] \arrow[d] & S^1 \arrow[d, "\mathrm{deg} \ d"] \\
                X & \arrow[l, "\delta_i"'] D_i \arrow[r, "\cong"]  & S^1
            \end{tikzcd} 
        \]
        commutes. By possibly enlarging $\Delta$, we can assume that the images of the $\overline{\delta}_j$ are contained in $\Delta$.
        
        In the construction of \cite[Lemma 2.23]{WiltonHall}, one first enlarges $X''$ to $X'''$ by adding some simply connected vertex spaces of $X_H$  to obtain $X'' \longhookrightarrow X''' \longhookrightarrow X_H$. In particular, the induced map $\pi_1(X'')\lrar \pi_1(X''')$ is an isomorphism. Then, one considers a collection of pairs $(\phi_{k, \o}^H, \phi_{k, \t}^H)$ of edge maps $\phi_{k, \o}^H\po \R_{\o}\lrar X_H$ and  $\phi_{k, \t}^H\po \R_{\t}\lrar X_H$ which are elevations of the incident and terminal edge maps of some edge space $S_k^1$ of $X$. Furthermore, these pairs $(\phi_{k, \o}^H, \phi_{k, \t}^H)$ will have the property that these are not edge maps of $X'''$ (such elevations are usually called {\it hanging elevations} of the precover $X'''$, as in \cite[Remark 2.18]{WiltonHall}). For each $k$, we denoted by $\R_{\o}$ and $\R_{\t}$ the domains of  $\phi_{k, \o}^H$ and $\phi_{k, \t}^H$, respectively (which are the universal covers of $S_k^1$). We fix a deck transformation $\tau\po \R_{\o}\lrar \R_{\t}$ so that the natural  diagram  
        \[
        \begin{tikzcd}
           \R_{\o} \ar[rr, "\tau"] \ar[dr] & &  \R_{\t} \ar[dl] \\
            & S_k^1&
        \end{tikzcd}
        \]
        commutes. Then $\overline X$ is constructed from $X'''$ by adding the edge space $\R_{\o}$ with incident and terminal edge maps given by $\phi_{k, \o}^H$ and $\phi_{k, \t}^H\circ \tau$ for each $k$.
        
        We denote by $\Ga$ the underlying graph of the splitting of $\overline{X}$. Notice that the underlying graph of $X''$ may  be smaller. We set $E_T\subset \E(\Ga)$ to be the edges of $\Ga(X''')$. We enlarge the splitting of $X''$ by adding trivial vertex groups and just assume that its underlying graph is also $\Ga$. So we view $\pi_1(X'')$ as the fundamental group of a graph of groups $\mathcal W$ whose underlying graph is $\Ga$. We denote by $T$ a spanning tree of the underlying graph $\Ga(X''')$ of $X'''$. By construction, $T$ is also a spanning tree of $\Ga$. 
    \end{step}
\begin{step}[The finite-sheeted precover $\widehat{X}$]
    By \cite[Proposition 3.4]{WiltonHall}, there exists an intermediate finite-sheeted precovering $\overline{X}\lrar \widehat{X}\lrar X$ satisfying the following properties for all sufficiently large positive integers $d$:
    \begin{enumerate}
        \item The underlying graph of $\widehat{X}$ is $\Ga$.
        \item The subcomplex $\Delta$ projects homeomorphically into $\widehat{X}$. 
        \item Each $\varepsilon_j'$ descends to a full elevation $\widehat{\varepsilon}_j\po \widehat{E}_j \lrar \widehat{X}$ with $\widehat{E}_j\lrar E_i$ being a covering of degree $d$. 
    \end{enumerate}
    Since $\Delta$ injects into $\widehat{X}$, we already know that $\overline{\delta}_j$ descends to a full elevation $\widehat{\delta}_j\po \overline{D}_j\lrar \widehat{X}$. We want to apply \cref{prop:monster} and prove that the natural map 
    \[
        \pi_1(X'')*\left(\coprod_{\mathcal L}\Z\right)\lrar \pi_1(\widehat{X})
    \]
    is injective and $L^2$-injective. Before this, we need to introduce more notation.  There is a natural bipartite structure of $\Ga$   given by the bipartite structure of the splittings of ICE groups. More precisely, $\V(\Ga)$ is split into disjoint sets $\V_{\o}$ and $\V_{\t}$ so that, for all $e\in \E(\Ga)$, $\o(e)\in \V_{\o}$, and $X_{\o(e)}''$ is a covering of $Y$; and, similarly, $\t(e)\in \V_{\t}$ and $X_{\t(e)}''$ is a covering of the torus $T^n$. We denote by $\mathcal Z$ the graph of groups corresponding to $\pi_1(\widehat{X})$, whose underlying graph is $\Ga$. At the end of Step 2, we defined the graphs of groups $\mathcal W$ and $\mathcal Z$, the spanning tree $T\subset \Ga$ and the subset of edges $\E(T)\subset E_T\subset \E(\Ga)$. Denote by $\mathcal{L}^{(0)}$ the collection of elements of $\pi_1(\widehat{X})$ that are represented by the images of the elliptic loops $\{\widehat{\varepsilon}_j\}$. For each $v\in \V_{\o}$, we define $\mathcal L_v$ to be the subset of $Z_v$ that contains $\mathcal L_v^{(0)}$ and the elements $\phi_{\o, e}(z_e)$ such that $\phi_{\o, e}(z_e)\notin W_v$. By construction, it is not hard to see that, up to a homotopy of $\widehat{X}$, we have that:
    
    \begin{itemize}
        \item[(a)] The subset of $\pi_1(\widehat{X})$ represented by the images of $\widehat{\epsilon_j}$ is exactly $\bigcup_{v\in \V_{\o}}\mathcal{L}_v^{(0)}$.
        \item[(b)] The subset of $\pi_1(\widehat{X})$ represented by the images of $\widehat{\delta_j}$ is  \[\{t_e: e\in \E(\Gamma)\setminus E_T\},\] where we view $\pi_1(\widehat{X})$ as in \cref{defn:tree_pi1} (relative to the spanning tree $T$).
    \end{itemize}
    
    Before applying \cref{prop:monster}, we observe that we can ensure that $\widehat{X}$ satisfies an additional property, on top of the three listed above. Our inductive hypothesis implies that the complex $\overline{X}_v$ is $L^2$-tame relative to $\mathcal{L}_v$ for each $v\in \V_{\o}$. Henceforth, with the same construction as in \cite[Proposition 3.4]{WiltonHall}, and by adequately replacing the notion of tameness by our notion of $L^2$-tameness, we could have ensured that the finite-sheeted precover $\widehat{X}$ satisfies the following additional property:
    \begin{itemize}
         \item[(4)] For each $v\in \V_{\o}$,  the natural map 
     \[\pi_1(\overline{X}_v)*\left(\coprod_{\mathcal{L}_v}\Z\right)\lrar \pi_1(\widehat{X}_v)\]
    is injective and $L^2$-injective. 
    \end{itemize}
    
    By applying \cref{prop:monster} to the subgraph of groups $\mathcal W\leq \mathcal Z$ introduced above (and keeping in mind remarks (a) and (b) above), the induced map 
    \begin{equation} \label{eq1} 
        \pi_1(X'')* \left( \coprod_{\mathcal L } \Z\right)\lrar \pi_1(\widehat{X}),
    \end{equation}
    is injective and $L^2$-injective.
    \end{step}
    \begin{step}[The finite-sheeted cover $\widehat{X}^{+}$]
    Finally, $\widehat{X}$ can be extended to a finite-sheeted covering $\widehat{X}^{+} \lrar X$ by adding additional vertex spaces glued along cylinders by \cite[Proposition 3.7]{WiltonHall}. Hence $\pi_1(\widehat{X})$ is the vertex group of a cyclic splitting of  $\pi_1(\widehat{X}^{+})$ and, by \cref{lem:MV}, the injective map \begin{equation} \label{eq2}\pi_1(\widehat{X})\lrar \pi_1(\widehat{X}^{+})\end{equation} is $L^2$-injective. 
    \end{step}
    
    We have gathered all the ingredients to prove that the finite-sheeted cover $\widehat{X}^{+}$ satisfies the fourth point of the $L^2$-tame property, namely that the induced map \[\pi_1(X_H)*\left(\coprod_{\mathcal L}\Z\right)\lrar \pi_1(\widehat{X}^{+}),\] is injective and $L^2$-injective. This is a direct consequence of the fact that the maps described  in Equations (\ref{eq0}), (\ref{eq1}) and (\ref{eq2}) are injective and $L^2$-injective. \qedhere
\end{proof}

\section{The Strengthened Hanna Neumann Conjecture} \label{sec:L2SHNC}
The purpose of this section is to explain how to obtain the Strengthened Hanna--Neumann conjecture for all the groups of \cref{cor:C}, that is, limit groups and finite graphs of free groups with cyclic edge groups that are hyperbolic relative to virtually abelian subgroups. The results that we state here are well-known.  In \cite[Sections 8--12]{Jai22}, Antolín and Jaikin-Zapirain explain how the $L^2$-Hall property implies the Strengthened Hanna Neumann conjecture (SHNC) for the class of hyperbolic limit groups. Here we  review their argument implementing  recent work of Minasyan \cite{Minasyan_WZ} and Minasyan--Mineh \cite{MinasyanMineh_QC} that will show that the $L^2$-Hall property implies the SHNC for the groups of \cref{cor:HNGFGCEG}. This proves that \cref{cor:C} follows from Theorems \ref{thm:A} and \ref{thm:B}.

We retain the following convention: if $G$ is a graph of free groups with cyclic edge groups, then $X = (X_v, X_e; \Gamma)$ will denote a corresponding graph of graphs with $S^1$ edge spaces such that $G = \pi_1 X$. However, in this subsection, $G$ is not always assumed to be a graph of free groups.

Before describing under which circumstances the $\Dg$-Hall property implies the SHNC, we record other properties that will be important for this.  Recall that a group $G$ is said to have the \textit{Wilson--Zalesskii property} if $G$ is residually finite and
\[
    \overline{U \cap V} = \overline{U} \cap \overline{V}
\]
for all finitely generated subgroups of $G$, where the closures are taken in the profinite completion $\widehat G$ of $G$. 

\begin{prop}\label{prop:props_of_GFG_and_limit}
    Let $G = \pi_1(X)$ be a finite graph of free groups with cyclic edge groups such that $\Phi_X$ is balanced and solvable (in the sense of \cref{def:balanced}). Then $G$ is
    \begin{enumerate}[label=(\arabic*)]
        \item\label{item:relhyp} hyperbolic relative to virtually abelian subgroups; 
        \item\label{item:locrelQC} locally relatively quasiconvex; 
        \item\label{item:virtCS} virtually compact special;
        \item\label{item:cosetsep} double coset separable and therefore has the Wilson--Zalesskii property;
        \item\label{item:L2Hall} $L^2$-Hall.
    \end{enumerate}
    These conclusions also hold if $G$ is a limit group.
\end{prop}
\begin{proof}
     Let $G = \pi_1(X)$ be a graph of free groups with cyclic edge groups such that $\Phi_X$ is balanced and solvable. By the main result of Richer's masters thesis \cite[Main Theorem 1.2.5]{Richer_thesis}, $G$ is hyperbolic relative to the generalised Baumslag--Solitar subgroups associated to the components of $\Phi_X$ (see \cite[Section 2.2]{Richer_thesis} for this version of the statement). Since $\Phi_X$ is balanced and solvable, all the edge weights are $\pm 1$ and each component of $\Phi_X$ has first Betti number at most $1$. Therefore, the associated generalised Baumslag--Solitar subgroups are all virtually isomorphic to $\Z$ or $\Z^2$. Hence \ref{item:relhyp} follows. Now   \ref{item:locrelQC}   follows from \ref{item:relhyp} and \cite[Corollary D]{BigdelyWise_QC}.

    Property \ref{item:virtCS} can be collected from either \ref{item:relhyp} (together the main results of \cite{HsuWise_graphOfFree} and \cite{Reyes_relhypcub}) or from \cite[Corollary 2.3]{MinasyanMineh_QC}. Minasyan--Mineh proved that both limit groups and  subgroup separable graphs of free groups with cyclic edge groups are double coset separable \cite[Theorem 2.2]{MinasyanMineh_QC}. Minasyan proved that double coset separability implies the Wilson--Zalesskii property \cite[Corollary 1.2]{Minasyan_WZ}. This implies  \ref{item:cosetsep}. Finally, \ref{item:L2Hall}  is exactly \cref{thm:GOFGWCEG}. 
    
    On the other hand, when $G$ is a limit group, (1) and (2)  follow from \cite[Theorem 0.3 and Proposition 4.6]{Dahmani_combination}; (3) is proved in \cite{HsuWise_graphOfFree}; (4) was explained in the previous paragraph; and the $L^2$-Hall property of (5) is a consequence of \cite[Corollary 28]{BrownKar2023quantifying} or of \cref{thm:L2tame}. \qedhere
\end{proof}

For the SHNC to hold for a class of groups, we must first ensure that the sum over the double cosets is finite. By \cite[Theorem 9.4]{Jai22}, this is the case for limit groups. We restate this theorem in enough generality so as to include graphs of free groups with cyclic edge groups such that \(\Phi_X\) is balanced and solvable. We also include a sketch of the proof, since it is essentially identical to the one of Antolín--Jaikin-Zapirain.

\begin{thm}[{\cite[Theorem 9.4]{Jai22}}]\label{thm:finNonAb}
    Let $G$ a torsion-free group that is hyperbolic relative to a family of virtually abelian subgroups and suppose that $G$ is locally relatively quasi-convex. If $U, V \leqslant G$ are finitely generated subgroups and $T$ is a complete set of $(U,V)$-double coset representatives, then there are only finitely many $t \in T$ such that $U \cap V^t$ is not virtually abelian. In particular, the sum $\sum_{t \in T} \rchi(U \cap V^t)$ is finite.
\end{thm}
\begin{proof}[Proof (sketch)]
    First, note that torsion-free abelian groups are of finite type. Since \(G\) is locally relatively quasiconvex, it follows that every finitely generated subgroup \(U\) of \(G\) is hyperbolic relative to subgroups of finite type, and therefore \(U\) is of finite type by \cite[Theorem 0.1]{Dahmani_typeFparabolics}. Moreover, the intersection of relatively quasiconvex subgroups is relatively quasiconvex, so it follows that \(\rchi(U \cap V^t)\) is defined and finite for all \(t \in T\). Hence, to prove the claim it suffices to show that \(\rchi(U \cap V^t) = 0\) for all but finitely many \(t \in T\).
    
    By \cite[Proposition 1.3]{Jai22}, there are only finitely many \(t \in T\) such that \(U \cap V^t\) contains a loxodromic element. Since \(G\) is torsion-free, this implies that \(U \cap V^t\) is contained in a parabolic for all but finitely many \(t \in T\). Therefore \(U \cap V^t\) is virtually abelian and in particular \(\rchi(U \cap V^t) = 0\) for all but finitely many \(t \in T\).
\end{proof}

For the remainder of the section, assume that the groups appearing are locally indicable and that their group algebras admit Hughes-free embeddings. The next important step in their argument is to reformulate the Strengthened Hanna Neumann Conjecture in terms of $\Dg$-Betti numbers of pairs of modules. First of all, recall the $\Dg$-Betti numbers of $K[G]$-modules introduced in \cref{DgBettieq2}. Let $M$ and $N$ be two left $K[G]$-modules. Form the $K[G]$ module $M \otimes_K N$, where the $G$-action on simple tensors is given by the diagonal action $g \cdot (m \otimes n) := (gm) \otimes (gn)$. The $n$-th $\Dg$-Betti number of the pair $(M,N)$ is defined to be
\[
    \beta_n^{K[G]}(M,N) := \beta_n^{K[G]}(M \otimes_K N).
\]

The next result is proved for limit groups in \cite[Proposition 2.8]{Jai22}. It also holds for the graphs of free groups with cyclic edge groups that we consider, with essentially the same proof.

\begin{prop}[{\cite[Proposition 8.2]{Jai22}}]\label{prop:module}
    Let $G$ be hyperbolic relative to virtually abelian subgroups and locally relatively quasiconvex, and suppose that all finitely generated subgroups of \(G\) are of finite type. Finally, assume that \(b_1^{K[U]}(U) = \rchi(U)\) for every finitely generated subgroup \(U \leqslant G\). Then for any pair of finitely generated subgroups $U,V \leqslant G$, we have
    \begin{equation} \label{accident2}
        \beta_1^{K[G]}(K[G/U], K[G/V]) = \sum_{t \in U \backslash G / V} \rchi(U \cap V^t).
    \end{equation}
    In particular, the Strengthened Hanna Neumann Conjecture for \(G\) is equivalent to having
    \[
        \beta_1^{K[G]}(K[G/U], K[G/V]) \leqslant b_1^{K[U]}(U) \, b_1^{K[V]}(V).
    \]
\end{prop}
\begin{proof}
    Let \(U,V \leqslant G\) be finitely generated subgroups and let \(T\) be a complete set of double \((U,V)\)-coset representatives. We always have the decomposition
    \[
        K[G/U] \otimes_K K[G/V] = \bigoplus_{t \in T} K[G/(U \cap V^t)],
    \]
    so using the fact that the \(\Tor_1^{K[G]}\) functor commutes with direct sums and taking \(\Dg\)-dimensions yields
    \[
        \beta_1^{K[G]}(K[G/U], K[G/V]) = \sum_{t \in T} \beta_1^{K[G]} (K[G/(U \cap V^t)]).
    \]
    Finally, \eqref{accident2} follows from the equation above and the fact that
    \[
        \beta_1^{K[G]}(K[G/U]) = b_1^{K[U]}(U)= \rchi(U)
    \]
    for any finitely generated subgroup \(U \leqslant G\) by \eqref{DgBettieq2} and \cref{prop:HFprops}\ref{item:Shapi0}.  
\end{proof}

\begin{rem}
    We claim that limit groups and finitely generated graphs of free groups with cyclic edge groups that are hyperbolic relative to virtually abelian subgroups satisfy the hypotheses of \cref{prop:module}. Both classes of groups are hyperbolic relative to virtually abelian subgroups and locally relatively quasiconvex by \cref{prop:props_of_GFG_and_limit}. Finitely generated subgroups of graphs of free groups with cyclic edge groups or limit groups are of finite type, a fact which can be deduced from \cref{prop:props_of_GFG_and_limit}\ref{item:locrelQC}, and using the fact that the peripheral subgroups are of finite type in each case. 
    
    Let \(G\) be a nontrivial graph of free groups with cyclic edge groups. Then \(b_n^{K[G]}(G)\) vanishes for \(n \neq 1\) by Chiswell's long exact sequence (\cref{thm:MV}). Since subgroups of graphs of free groups with cyclic edge are again graphs of free groups with cyclic edge groups, it follows that \(\rchi(U) = b_n^{K[U]}(U)\) for any finitely generated subgroup of a graph of free groups with cyclic edge groups \(G\).

    Now suppose that \(G\) is a limit group. Then \(b_n^{K[G]}(G)\) is only nonzero for \(n = 1\), which can be seen in several ways. For example, one can combine \cite[Corollary B]{BridKoch_volumeGradient} with the fact that \(b_n^{K[G]}(G) \leqslant b_n(G;K)\) (see \cite[Corollary 1.6]{JaikinZapirain2020THEUO}, where the result is stated for \(K = \C\), but the proof is identical for any field \(K\)). It is also possible to derive this fact from the result that limit groups are free-by-(torsion-free nilpotent) \cite{Kochloukova_subdirect}. Since all finitely generated subgroups of limit groups are limit groups (this immediate from the characterisation of \cref{thm:limit_are_subgp_of_ICE}). This implies that limit groups satisfy the hypotheses of \cref{prop:module}.
\end{rem}

The last key result from the article of Antolín--Jaikin-Zapirain, is stated for hyperbolic limit groups. The proof in the general setting is nearly identical; we reproduce it here for the sake of completeness and in order to elucidate the differences in the relatively hyperbolic context. 

\begin{prop}[{\cite[Proposition 11.1]{Jai22}}]\label{prop:b20}
    Let $G$ be a locally indicable group that is hyperbolic relative to virtually abelian subgroups. Additionally, assume $G$ that is double coset separable, $\Dg$-Hall, and has the Wilson--Zalesskii property. Let $U,V \leqslant G$ be finitely generated subgroups. Then there exists a normal finite-index subgroup $H \trianglelefteqslant G$ such that $\beta_1^{K[G]}(N) = 0$, where $N$ is the kernel of the map
    \[
        f \colon K[G/U] \otimes_K K[G/V] \lrar K[G/U] \otimes_K K[G/VH].
    \]
\end{prop}
\begin{proof}
    By \cref{thm:finNonAb}, there are only finitely many double cosets $UtV$ such that $U \cap V^t$ is not virtually abelian. Let $H_0 \leqslant G$ be a finite-index subgroup separating these double cosets $UtV$.

    For each $t$ such that $UtV$ is not virtually abelian, let $A_t \trianglelefteqslant G$ be a finite-index normal subgroup such that $U\cap V^t$ is $\Dg$-independent in $(U\cap V^t)A_t$. By \cite[Corollary 10.4]{Jai22} (which holds for groups with the Wilson--Zalesskii property), there is a finite-index normal subgroup $H_t \trianglelefteqslant G$ such that $U \cap (H_tV)^t \leqslant (U \cap V^t)A_t$ . Set $H = H_0 \cap \bigcap_t A_t$, where $t$ runs over the double coset representatives such that $UtV$ is not virtually abelian.

    Let $T$ be a set of $(U,VH)$ double coset representatives containing $1$, which extends to $T'$, a set of $(U,V)$-double coset representatives. Since $H_0$ separates the non-virtually abelian $(U,V)$-double cosets, it follows that if $x \in T' \smallsetminus T$ then $UxV$ is virtually abelian. Let $\pi \colon T' \lrar T$ be a set-theoretic map with the property that $U\pi(t)VH = UtVH$ for all $t \in T'$. In general we have the $K[G]$-module decomposition
    \[
        K[G/U] \otimes_K K[G/V] \cong \bigoplus_{t \in T'} K[G] (U \otimes tV).
    \]
    However, in order to analyse the kernel of $f$ more easily, it is useful to modify the complement of $\oplus_{t \in T} K[G](U \otimes tV)$ in $K[G/U] \otimes_K K[G/V]$ and obtain the following decomposition:
    \[
        K[G/U] \otimes_K K[G/V] \cong \bigoplus_{t \in T} K[G] (U \otimes tV) \oplus \bigoplus_{t \in T' \smallsetminus T} K[G] (U \otimes (tV - \pi(t)V)).
    \]
    Let $I_t$ denote the kernel of the map $K[G](U \otimes tV) \lrar K[G](U \otimes tVH)$. Then the kernel of $f$ has the decomposition
    \[
        \ker f \cong \bigoplus_{t \in T} I_t \oplus \bigoplus_{t \in T' \smallsetminus T} K[G] (U \otimes (tV - \pi(t)V)).
    \]
    By exactly the same proof as in \cite[Proposition 11.1]{Jai22}, $\beta_1^{(2)}(I_t) = 0$ for each $t \in T$. On the other hand, we have isomorphisms
    \[
        K[G] (U \otimes (tV - \pi(t)V)) \cong K[G] (U \otimes (tV - \pi(t)V)) \cong K[G/(U \cap V^t)].
    \]
    So it suffices to compute $\beta_1^{K[G]}(K[G/A])$ where $A$ is a virtually abelian group, which is equal to $b_1^{K[A]}(A) = 0$ by \eqref{DgBettieq2} and \cref{prop:HFprops}\ref{item:Shapi0}. \qedhere
\end{proof}

With all of this in place, the proof that the $L^2$-hall property implies the SHNC proceeds exactly as in Section 12.1 of \cite{Jai22}, where the following is shown.

\begin{lem}[{\cite[Section 12.1]{Jai22}}]\label{lem:AJZ_section_12_1}
    Let \(G\) be a torsion-free virtually compact special group such that \(b_n^{K[G]}(G) = 0\) for all \(n \neq 1\), and let \(U, V \leqslant G\) be finitely generated subgroups. If there is a finite-index normal subgroup \(H \trianglelefteqslant G\) such that 
    \[
        \beta_1^{K[G]}(M) = 0 = \beta_1^{K[G]}(K[G/U],N),
    \]
    where
    \[
        M = \ker(K[G/U] \lrar K[G/UH]) \quad \text{and} \quad N = \ker(K[G/V] \lrar K[G/VH]),
    \]
    then \(\beta_1^{K[G]}(K[G/U], K[G/V]) \leqslant b_1^{K[U]}(U) b_1^{K[V]}(V)\).
\end{lem}

The details of the argument are rather technical and rely on the theory of acceptable modules over twisted group rings, which is developed in \cite[Section 6]{Jai22}. All the results of \cite[Section 6]{Jai22} hold for all virtually compact special groups, except for \cite[Proposition 6.4]{Jai22}, which is stated only for limit groups. However, the only properties of limit groups that are used in the proof are that they are virtually compact special and that their \(L^2\)-Betti numbers vanish in all dimension other than \(1\). We are now ready to conclude that the SHNC holds for graphs of free groups with cyclic edge groups that are hyperbolic relative to virtually abelian subgroups limit groups.

\begin{cor}\label{cor:HNGFGCEG}
    Let $G = \pi_1(X)$ be a graph of free groups with cyclic edge groups such that $\Phi_X$ is balanced and solvable (equivalently, such that $G$ is hyperbolic relative to virtually abelian subgroups) or let $G$ be a limit group. Then $G$ satisfies the Strengthened Hanna Neumann Conjecture.
\end{cor}
\begin{proof}
    By \cref{prop:props_of_GFG_and_limit} and \cref{prop:b20}, for any pair of finitely generated subgroups \(U,V \leqslant G\), there is a normal finite-index subgroup \(H \trianglelefteqslant G\) such that \(\beta_1^{(2)}(L) = 0\), where \(L\) is the kernel of
    \[
        \Q[G/U] \otimes_\Q \Q[G/V] \lrar \Q[G/U] \otimes_\Q \Q[G/VH].
    \]
    Since \(L = \Q[G/U] \otimes_\Q N\), where \(N = \ker(\Q[G/V] \lrar \Q[G/VH])\), it follows immediately that \(\beta_1^{\Q[G]}(\Q[G/U]) = 0\).

    We may also choose \(H\) deep enough so that \(U\) is \(L^2\)-independent in \(UH\). Let \(M\) be the kernel of \(\Q[G/U] \lrar \Q[V/UH]\). By \cite[Proposition 4.22]{Jai22}, this implies that \(\beta_1^{\Q[G]}(M) = 0\). By \cref{prop:module,lem:AJZ_section_12_1}, we conclude that \(G\) satisfies the SHNC. \qedhere
\end{proof}

\bibliographystyle{alpha}
\bibliography{bib}

\end{document}